\documentclass[11pt,a4paper,reqno]{amsart}            

\usepackage[english]{babel}
\usepackage{amssymb,amsmath,amsthm}
\usepackage{mathrsfs}
\usepackage{enumitem}
\usepackage[immediate,debrief]{silence}
\usepackage[pagebackref]{hyperref}

\usepackage{color}

\usepackage{scrtime}

\hypersetup{
 colorlinks,
 linkcolor={blue!90!black},
 citecolor={red!80!black},
 urlcolor={blue!50!black}
}
\usepackage{tikz}
\usepackage{tikz-cd}
\usepackage{graphicx}
\usetikzlibrary{shapes.geometric,arrows,fit,matrix,positioning,trees,snakes}
\tikzset
{
    treenode/.style = {circle, draw=black, align=center, minimum size=1cm},
    subtree/.style  = {isosceles triangle, draw=black, align=center, minimum height=0.5cm, minimum width=1cm, shape border rotate=90, anchor=north}
}

\renewcommand\mathcal{\mathscr}
\theoremstyle{plain}
\newtheorem{theorem}{Theorem}[section]
\newtheorem*{theorem*}{Theorem}
\newtheorem{lemma}[theorem]{Lemma}
\newtheorem*{lemma*}{Lemma}

\newtheorem{proposition}[theorem]{Proposition}

\theoremstyle{remark}
\newtheorem{remark}[theorem]{Remark}
\newtheorem*{remark*}{Remark}

\theoremstyle{definition}
\newtheorem{definition}[theorem]{Definition}
\newtheorem*{definition*}{Definition}

\theoremstyle{example}

\newtheorem*{example*}{Example}

\numberwithin{equation}{section}

\newcount\quantno
\everydisplay{\quantno=0}\everycr{\quantno=0}
\newcommand\quant{\advance\quantno by1
                      \ifnum\quantno=1\qquad\else\quad\fi\forall }
\newcommand\itemno[1]{(\romannumeral #1)}

\newcommand\rest[1]{\kern-.1em
          \lower.5ex\hbox{$\scriptstyle #1$}\kern.05em}

\renewcommand\mod[1]{\vert{#1}\vert}
\newcommand\bigmod[1]{\bigl\vert{#1}\bigr|}

\newcommand\norm[2]{{\Vert{#1}\Vert_{#2}}}

\newcommand\bignorm[2]{\left.{\bigl\Vert{#1}\bigr\Vert_{#2}}\right.}
\newcommand\bignormto[3]{\left.{\bigl\Vert{#1}\bigr\Vert_{#2}^{#3}}\right.}

\newcommand\bigopnorm[2]{\big|\!\big|\!\big| {#1} \big|\!\big|\!\big|_{#2}}
\newcommand\bigopnormto[3]{\big|\!\big|\!\big| {#1} \big|\!\big|\!\big|_{#2}^{#3}}

\newcommand\wrt{\,\text{\rm d}}

\newcommand\BN{\mathbb{N}}

\newcommand\BZ{\mathbb{Z}}

\newcommand\cA{\mathcal{A}}   
  
  \newcommand\fB{\mathfrak{B}}

\newcommand\cE{\mathcal{E}}   
 
  \newcommand\fF{\mathfrak{F}} 
 
  \newcommand\fG{\mathfrak{G}} 
 
  \newcommand\fH{\mathfrak{H}}

\newcommand\cJ{\mathcal{J}}

\newcommand\cM{\mathcal{M}}     
\newcommand\cN{\mathcal{N}}

  \newcommand\fS{\mathfrak{S}}  
  \newcommand\fT{\mathfrak{T}}  
  \newcommand\fU{\mathfrak{U}}  
  \newcommand\fV{\mathfrak{V}}

\newcommand\al{\alpha}
\newcommand\be{\beta}
\newcommand\ga{\gamma}    
\newcommand\de{\delta}
  \newcommand\vep{\varepsilon}

\newcommand\la{\lambda}   
\newcommand\om{\omega}    \newcommand\Om{\Omega}  
\newcommand\si{\sigma}    

\newcommand\vp{\varphi}
\newcommand\vr{\varrho}

\newcommand\funnyk{k\hbox to 0pt{\hss\phantom{g}}}

\newcommand\lu[1]{L^1(#1)}
\newcommand\lp[1]{L^p(#1)}

\newcommand\ly[1]{L^\infty(#1)}
\newcommand\lorentz[3]{L^{#1,#2}(#3)}

\newcommand\wt{\widetilde}
\newcommand\whH{\widehat{\phantom{G}}\hbox to 0pt{\hss $H$}}

\newcommand\emspace{\hbox to 6pt{\hss}}
\newcommand\ds{\displaystyle}

\newcommand\rmi{\hbox{\rm (i)}}
\newcommand\rmii{\hbox{\rm (ii)}}

\newcommand\ioty{\int_0^{\infty}}

\newcommand\One{{\mathbf{1}}}

\newcommand\diam{\mathrm{diam}}

\newcommand\pab{\tau}

\DeclareSymbolFont{EUEX}{U}{euex}{m}{n}

\DeclareSymbolFont{euexlargesymbols}{U}{euex}{m}{n}
\DeclareMathSymbol{\intop}{\mathop}{euexlargesymbols}{"52}
     \def\int{\intop\nolimits}

\DeclareSymbolFont{euexsymbols}     {U}{euex}{m}{n}
\DeclareMathSymbol{\smallint}{\mathop}{euexsymbols}{"52}

\begin{document}

\title[Maximal operators]{Hardy--Littlewood maximal operators \\ on trees with bounded geometry
}



\thanks{
Levi, Santagati and Vallarino are members of the Gruppo Nazionale per l'Analisi Matematica, la Probabilit\`a e le loro Applicazioni (GNAMPA) of the
Istituto Nazionale di Alta Matematica (INdAM).
Levi, Santagati and Vallarino have been supported in part by the Project ``Harmonic analysis on continuous and discrete structures'' (bando Trapezio 
Compagnia di San Paolo CUP E13C21000270007).
Levi has been supported in part by the Generalitat de Catalunya (grant 2021 SGR 00071), the Spanish Ministerio de Ciencia e Inn\'ovacion (project PID2021-12315NB-100).
}

\author[M.\ Levi]{Matteo Levi}
\address[Matteo Levi]{Facultat de Matem\`atiques i Inform\`atica, Universitat de Barcelona, Gran Via de les Corts Catalanes 585, 08007, Barcelona, Spain.}
\email{matteo.levi@ub.edu}

\author[S. Meda]{Stefano Meda}
\address[Stefano Meda]{Dipartimento di Matematica e Applicazioni \\ Universit\`a di Milano-Bicocca\\
via R.~Cozzi 53\\ I-20125 Milano\\ Italy}
\email{stefano.meda@unimib.it}

\author[F.\ Santagati]{Federico Santagati} \address[Federico Santagati]{Dipartimento di Matematica e Applicazioni \\ Universit\`a di Milano-Bicocca\\
via R.~Cozzi 53\\ I-20125 Milano\\ Italy}
\email{federico.santagati@unimib.it}

\author[M. Vallarino]{Maria Vallarino}
\address[Maria Vallarino]{Dipartimento di Scienze Matematiche ``Giuseppe Luigi Lagrange''  \\
Politecnico di Torino\\ corso Duca degli Abruzzi 24\\ 10129 Torino\\ Italy}
\email{maria.vallarino@polito.it}

\begin{abstract}
In this paper we study the $L^p$ boundedness of the centred and the uncentred Hardy--Littlewood maximal operators on the class~$\Upsilon_{a,b}$, 
$2\leq a\leq b$, of trees with $(a,b)$-bounded geometry.  We find the sharp range of $p$, depending on $a$ and $b$, 
where the centred maximal operator is bounded on $\lp{\fT}$ for all $\fT$ in $\Upsilon_{a,b}$.
We show that there exists a tree in $\Upsilon_{a,b}$ for which the uncentred maximal function is bounded on $L^p$ if and only if $p=\infty$.  
We also extend these results to graphs which are strictly roughly isometric, in the sense of Kanai, to trees in the class $\Upsilon_{a,b}$.
\end{abstract}

\date{\today, \thistime. Preliminary version}

\maketitle
\begin{center}
{\textsl{This paper is dedicated to the memory of A. M.}}
\end{center}
\vskip0.8cm

\section{Introduction}
\label{s: Introduction}

The purpose of this paper is to prove sharp $L^p$ bounds for the centred and the uncentred Hardy--Littlewood maximal operators on 
a class of trees with \textit{bounded geometry}.

Let $\fT$ be a locally finite tree, i.e. a connected graph with no loops, in which every vertex $x$ has a finite number $\nu(x)$ of neighbours; $\nu(x)$ 
is called the \textit{valence} of $x$.  We endow $\fT$ with the standard graph distance $d$ and the set of its vertices with the counting measure $\mu$.
We denote by $B_r(x)$ the (closed) ball with centre $x$ and radius $r$, i.e.  $B_r(x) := \{y\in \fT: d(x,y) \leq r\}$.  
For the sake of notational convenience, we write $\mod{E}$ instead of $\mu(E)$ for every subset $E$ of~$\fT$.   

For each function $f$ on $\fT$, we consider its \textit{centred} and its \textit{uncentred} (HL) maximal functions, defined by 
$$
\cM f(x)
:= \sup_{r\in \BN} \, \frac{1}{\mod{B_r(x)}} \, \int_{B_r(x)} \mod{f} \wrt \mu
\quad\hbox{\textrm{and}}\quad 
\cN f(x)
:= \sup_{B \ni x} \, \frac{1}{\mod{B}} \, \int_{B} \, \mod{f} \wrt \mu,
$$
respectively.

Sharp $L^p$ bounds for $\cM$ and $\cN$ are well known for the homogeneous tree $\fT_b$, $b\geq 2$, i.e. the tree such that $\nu(x) = b+1$ for 
every vertex $x$.  Specifically, the weak type $(1,1)$ boundedness of $\cM$ has been established independently by A.~Naor and T.~Tao \cite[Theorem~1.5]{NT},
and M.~Cowling, Meda and A.~Setti \cite[Theorem~3.1]{CMS1}, and is also a consequence of the analysis of the Green operator 
for certain random walks on trees performed in \cite{RT}.  Furthermore, the operator $\cN$ is of restricted weak type $(2,2)$, 
hence, by a simple interpolation argument with the trivial $\ly{\fT_b}$ bound, it is bounded on $\lp{\fT_b}$ for $p > 2$.  
The restricted weak type $(2,2)$ estimate is a direct consequence of a result of A.~Veca~\cite[Theorem~5.1]{V}, 
who proved that the modified maximal operator $\cM_2$ (see \eqref{f: modified si} below for the definition of $\cM_\si$, $\si>1$), 
which dominates $\cN$, is of restricted weak type $(2,2)$.  Levi and Santagati \cite{LS} proved that Veca's results for the operator $\cM_2$
is optimal in the scale of Lorentz spaces and, more generally, investigated the mapping properties of $\cM_\si$ between Lorentz spaces.  
Their couterexamples may be adapted to prove that $\cN$ is unbounded on $\lp{\fT_b}$
for all $p$ in $[1,2]$.  

J.~Soria and P.~Tradacete \cite[Theorem~4.1]{ST} showed that the method of proof of~\cite[Theorem~1.5]{NT} 
is amenable to generalisations to certain nonhomogeneous trees, and even graphs, with bounded valence function.  Indeed, they proved that $\cM$
is of weak type $(1,1)$ provided a certain ``technical" condition is satisfied.  Though this condition is remarkably general and very interesting, 
quite a lot of work is needed to check that it is satisfied in concrete cases.  Further work in this direction was done by 
S.~Ombrosi et al. \cite{ORS,OR}, who proved an analogue on the $k$-rooted tree of the Fefferman--Stein inequality
and considered the problem of establishing weighted $L^p$ estimates for the centred maximal operator. 

We also mention recent work of D.~Kosz \cite{K1,K2,K3}, who produced examples of discrete metric measure spaces, where $\cM$ and/or $\cN$ are bounded
on $L^p$ for every $p$ in any preassigned interval of the form $(p_0,\infty]$, where $p_0 \geq 1$.  His investigation includes also finer statements
involving restricted weak type estimates, or weak type estimates at the endpoint $p_0$.
However, it is worth observing that such metric measure spaces are not locally doubling. 

Notice that there are examples of trees $\fT$ with unbounded valence function~$\nu$ such that $\cM$, and \textit{a fortiori} $\cN$, 
is unbounded on $L^p(\fT)$ for every $p\in [1,\infty)$: see the example at the beginning of Section~\ref{s: Boundedness of cM}.  It may 
be interesting to mention that in this example the counting measure is not locally doubling.

The considerations above suggest that it may be interesting to find out reasonably wide classes of trees 
or graphs, where sharp $L^p$ bounds for $\cM$ and/or $\cN$ can be proved.  Our focus will be on the class $\Upsilon_{a,b}$ of all trees~$\fT$ 
such that $3\leq a+1\leq \nu(x) \leq b+1$ for every vertex $x$ (we shall always assume that $a+1$ and $b+1$ are attained).  
Trees in $\Upsilon_{a,b}$ are said to be of \textit{$(a,b)$-bounded geometry}: they have exponential volume growth, and the counting measure~$\mu$ 
is locally, but not globally, doubling.  Notice that if $a<b$, then the group of isometries of $\fT$ may be very small, and the group theoretic techniques 
from abstract Harmonic Analysis not available (see \cite{MZ,FTN}, and the references therein, for more on the group of isometries of homogeneous trees).  
The main question is whether sharp $L^p$ bounds for $\cM$ and $\cN$ hold for all trees in $\Upsilon_{a,b}$,
and, in case this fails, which subclasses of $\Upsilon_{a,b}$ share similar $L^p$ boundedness properties for $\cM$ and $\cN$.  We set 
\begin{equation} \label{f: def tau}
\pab
:= \log_a b;
\end{equation}
note that $\pab$ is always $\geq 1$, and $\pab \leq 2$ if and only if $b\leq a^2$. Since, for $a=b$, the class $\Upsilon_{a,b}$ only contains 
the homogeneous tree $\fT_b$, where a complete description of the $L^p$ boundedness properties of $\cM$ and $\cN$ is already available, 
we shall always assume that $a<b$, i.e. $\pab>1$, unless otherwise specified. 

We prove that if $a<b\leq a^2$, equivalently if $1<\pab\leq 2$, then $\cM$ is of restricted weak type $(\pab,\pab)$, and it is bounded on $\lp{\fT}$ 
for $p>\pab$ (see Theorem~\ref{t: main ab}~\rmi).  This result is proved by estimating $\cM$ with an appropriate convolution operator on the homogeneous
tree $\fT_b$, to which the sharp form of the Kunze--Stein phenomenon (see \cite[Theorem~1]{CMS1}) applies.  

The result is optimal, in the sense that if $a<b\leq a^2$, then there exists a tree~$\fS_{a,b}$ in the class $\Upsilon_{a,b}$ such that $\cM$ is unbounded 
on $\lp{\fS_{a,b}}$ when $1\leq p <\pab$ (see Proposition~\ref{p: counterexample Stromberg}~\rmi).  The tree $\fS_{a,b}$ is described just above 
Proposition~\ref{p: counterexample Stromberg}.  Furthermore, if $q$ is a number such that $1<q< \pab$, then we can find a number $s \leq q$, 
arbitrarily close to $q$, and a tree $\fT$ in $\Upsilon_{a,b}$ such that $\cM$ is bounded on $\lp{\fT}$ if and only if $p>s$ 
(see Proposition~\ref{p: denseness}).

In view of these results, it would be reasonable to conjecture that if, instead, $b> a^2$, i.e. $\pab >2$, then $\cM$ is bounded on $L^p$
at least for some finite, large enough, $p$.  To our surprise, this turns out to be false.  Indeed, we can prove that if $\pab>2$, 
then $\cM$ is bounded on $\lp{\fS_{a,b}}$ if and only if $p =\infty$ (see Proposition~\ref{p: counterexample Stromberg}~\rmii).  

Instead of assumptions on the geometry of the tree based on pointwise bounds on the valence, as in Theorem~\ref{t: main ab}~\rmi, it may be interesting
to formulate $L^p$ boundedness results in terms of lower bounds on the volume growth of geodesic balls.  In particular, we complement 
Theorem~\ref{t: main ab}~\rmi, by showing that if $\fT$ is in~$\Upsilon_{a,b}$ and there exist 
positive constants $\al$ and $C$ such that $\bigmod{B_r(x)} \geq C\, \al^r$ for all nonnegative integers $r$ and all points
$x$ in~$\fT$, and $b\leq \al^2$, then $\cM$ is bounded on $\lp{\fT}$ for all $p> \log_\al b$ 
(see Theorem~\ref{t: main ab}~\rmii).  Notice that $\al$ must be $\geq a$, so that $\log_\al b \leq \pab$ and the additional assumption
on the lower bound on the volume growth of $\fT$ yields a wider range of $p$'s for which $\cM$ is bounded on $\lp{\fT}$.  

A related problem consists in establishing weak type $(1,1)$ boundedness results under assumptions on the volume growth of geodesic balls.
In this paper we shall draw an interesting consequence of \cite[Lemma~5.1]{NT} and \cite[Theorem~4.1]{ST}, and show that if $\bigmod{B_r(x)}$ is of the 
same order of magnitute as~$\al^r$, uniformly in $r$ and $x$, then $\cM$ is of weak type $(1,1)$ (see Theorem~\ref{t: weak11}).

We point out that the $L^p$ boundedness properties of $\cM$ on a specific tree in $\Upsilon_{a,b}$ do not solely depend on the ratio $b/a$.  
For instance, no matter how big $b/a$ is, the operator $\cM$ is of weak type $(1,1)$ 
on the semi-homogeneous tree $\fV_{a,b}$ with valences $a+1$ and $b+1$ (see Section~\ref{s: Boundedness of cM} for the definition of $\fV_{a,b}$).  
This result is a straightforward consequence of Theorem~\ref{t: weak11} mentioned above. 

Rather, the boundedness properties of $\cM$ are, in a way or another, related to the fact that  
for points $x$, $y$ in $\fT$ at distance $r$, the balls $B_r(x)$ and $B_r(y)$ may have very different volumes.  
In fact, for \lq\lq many" pairs of points~$x$ and~$y$, the ratio $\mod{B_r(x)}/\mod{B_r(y)}$ can be of the same order of magnitude as~$(b/a)^r$,
as $r$ becomes large.   This seems also related to the failure of the ``equidistant comparability condition" (\textit{ECP} for short) considered in 
\cite[Definition~2.6]{ST}.  In general, the $L^p$ boundedness of $\cM$ and $\cN$ depends in a very subtle way 
on the geometry of a given tree, or graph, and, more specifically, on the location of vertices with different valences within the tree.  

As far as the uncentred maximal function is concerned, we prove (see Theorem~\ref{t: unboundedness nc}) that as soon as $a<b$, 
then there is a tree $\fF_{a,b}$ (introduced just above Theorem~\ref{t: unboundedness nc}) in the class $\Upsilon_{a,b}$ such that 
$\cN$ is bounded on $\lp{\fF_{a,b}}$ if and only if $p=\infty$.  This is even more striking, for we prove that the centred maximal operator 
$\cM$ is of weak type $(1,1)$ on $\fF_{a,b}$ (see Theorem~\ref{t: unboundedness nc}~\rmii).   

Finally, we investigate the robustness of the results obtained so far.  A natural issue is whether the $L^p$ boundedness 
of either $\cM$ or $\cN$ on a tree~$\fT$ is preserved if we modify $\fT$, but not too much.  In Section~\ref{s: Robustness} we consider graphs $\fG$ 
that are strictly roughly isometric to trees $\fT$ in the class $\Upsilon_{a,b}$, and prove that if either $\cM$ or $\cN$ is bounded on $\lp{\fT}$, 
then the corresponding maximal operator is bounded on $\lp{\fG}$ (see Theorem~\ref{t: quasi iso}, which contains a slightly more general result).  
A strict rough isometry between $\fT$ and $\fG$ is a (possibly non-surjective) map $\vp: \fT\to \fG$, 
which approximately preserve distances (see Definition~\ref{def: RI}) and such that $\vp(\fT)$ is at controlled distance from each point of $\fG$.

Finally, it is worth recalling that homogeneous trees $\fT_b$, $b\geq 2$, are discrete counterparts of symmetric spaces of the noncompact type $X$.  
Although $\fT_b$ is not roughly isometric in the sense of Kanai to any symmetric space of the noncompact type, quite often the analysis
at infinity of many operators on~$X$ and of related operators on $\fT_b$ show similar features and can be treated by similar methods.  
The relationship between symmetric spaces and homogeneous trees appears to be even stronger by looking at Harmonic Analysis on 
noncompact semisimple Lie groups on the one hand and on the group of isometries of homogeneous trees on the other.  A deep insight into a specific
problem in one setting often gives hints to treat related problems on the other.  
In particular, the aforementioned results for $\fT_b$ are versions in the discrete case of homologous results on $X$.  
Specifically, the weak type $(1,1)$ result for the (analogous of) the maximal operator
$\cM$ is proved in \cite{Str}, the restricted weak type $(2,2)$ estimates for $\cM_2$ 
in the case where $X$ has real rank one was proved by A.~Ionescu \cite{I2}.  Since $\cM_2$ dominates $\cN$, the latter operator is also
of restricted weak type $(2,2)$, and, interpolating with the trivial bound on $\ly{X}$, bounded on 
$\lp{X}$ for $p>2$.  The $L^p$ boundedness of $\cN$ in the higher rank case for $p>2$ is proved in \cite{I1}.  

Interesting related results on a class of cusped manifolds with nonconstant curvature can be found in \cite{L1,L2,L3}.  In these papers, H.-Q.~Li exhibits 
examples of manifolds $M$ with cusps, where the range of $p$'s for which $\cM$ and $\cN$ are bounded on $L^p$ can be any interval of the form 
$(p_0,\infty)$, or $(2p_0,\infty)$, respectively, where $p_0$ is a number in $[1,\infty)$, depending on certain parameters related to 
the geometry of $M$.
Our results are reminiscent of his. However, it is worth observing that, except for a few special cases, the metric measure spaces considered 
by H.-Q.~Li have, in a way or another, unbounded geometry.  

The paper is organised as follows.  Section~\ref{s: Preliminaries} is devoted to background material and preliminary estimates.  The centred and the 
uncentred maximal operators on trees in the class $\Upsilon_{a,b}$ are studied in Sections~\ref{s: Boundedness of cM} and~\ref{s: Unboundedness nc}, 
respectively.  Finally, Section~\ref{s: Robustness} is devoted to the proof that strict rough isometries between graphs satisfying mild additional 
assumptions preserve the $L^p$ boundedness of the centred and the uncentred maximal operators.

We use the ``variable constants convention'', and denote by~$C$ a constant whose actual value may vary from place to place and which
might depend on factors quantified before its occurrence, but not on factors quantified after.

\section{Preliminaries}
\label{s: Preliminaries}

Throughout this paper $\fT$ will denote a \textit{tree}, i.e., a connected graph with no loops.  We identify a tree with the set of its vertices, 
and declare that two points on $\fT$ are \textit{neighbours} if and only if they are connected by an edge.  We shall denote 
by $\nu(x)$ the \textit{valence} of the vertex $x$, i.e. the number of neighbours of~$x$.  We shall always assume that $2\leq \nu(x) < \infty$.  

The tree~$\fT$ carries a natural \textit{distance}~$d$: for each pair of vertices~$x$ and~$y$ we denote by $d(x,y)$~the number of edges between 
$x$ and $y$. 
It will be often convenient to consider a fixed but arbitrary reference point~$o$ in~${\fT}$; then we shall write~$\mod{x}$ for~$d(x,o)$.
For each $x$ in $\fT$ and each nonnegative integer $r$, the sphere $S_r(x)$ with centre $x$ and radius $r$ is defined as follows
$$
S_r(x)
:= \{y\in \fT: d(x,y) =r \}.  
$$
Then $\ds B_r(x) = \bigcup_{j=0}^r \, S_j(x)$ is the (closed) ball with centre $x$ and radius $r$.

Between any two points $x$ and $y$ in $\fT$, such that $d(x,y)=n$, there is a unique \textit{geodesic path} of the form
$x_0,x_1,\ldots,x_n$, where $x_0=x$, $x_n=y$, and $d(x_i,x_j)=\mod{i-j}$ whenever $0\leq i,j \leq n$.
A \textit{geodesic ray} $\ga$ in~$\fT$ is a one-sided sequence $\{\ga_n:n\in\BN\}$ of points of~$\fT$ such that
$d(\ga_i,\ga_j)=\mod{i-j}$ for all nonnegative integers $i$ and $j$.  We say that {\it $x$ lies on~$\ga$}, and write $x\in \ga$, if
$x=\ga_n$ for some $n$ in~$\BN$. Geodesic rays $\{\ga_n : n\in \BN\}$ and $\{\ga'_n : n\in \BN\}$ are identified if there exist integers
$i$ and $j$ such that $\ga_n=\ga'_{n+i}$ for all $n$ greater than $j$. This identification is an equivalence relation. 
In every equivalence class, there is a unique geodesic ray 
starting at $o$. 


We describe the \textit{horocycles} in~$\fT$.  Given a geodesic ray $\om$, we define the height function~$h_\om$ by the formula
$$
h_\om(x) = \lim_{m\to \infty} \big(m - d(x,\om_m)\big);
$$
$h_\om$ is the discrete analogue of the Busemann function in Riemannian geometry.  Then the \textit{$\om$-horocycles} are the level
sets of the function~$h_\om$; for $h$ in~$\BZ$ we set
$$
\fH_h^\om
:= \big\{x\in \fT : h_\om(x)=h\big\},
$$
and $\fT$ decomposes disjointly:
$$
\fT = \bigcup_{h\in \BZ} \, \fH_h^\om.
$$
If $\om$ and $\om'$ are equivalent, i.e., for some integers $i$ and $j$, $\om_n=\om'_{n+i}$ when $n\geq j$, then
$$
\begin{aligned}
h_\om(x)
& = \lim_{m\to \infty} \big(m-d(x,\om_m)\big) \\
& = \lim_{m\to \infty}\big( (m+i) - d(x,\om'_{m+i})\big) -i \\
& = h_{\om'}(x) - i, 
\end{aligned}
$$
and it follows that $\fH_h^\om = \fH_{h+i}^{\om'}$.  In particular, equivalent geodesic rays determine the same family of horocycles,
and the same horocyclical decomposition of $\fT$, even though the indices depend on the choice of the representative in the
equivalence class. 

\textit{Caveat}:
given a tree $\fT$, we implicitly assume that we have chosen a reference point~$o$ and a geodesic ray $\om$ in $\fT$, and consider the 
associated horocyclic decomposition of $\fT$.  For notational convenience we shall denote simply by~$h$ the height function 
relative to $\om$ and by $\fH_j$ the corresponding horocycles.   

Notice that if $x\in \fH_j$, then $\nu(x)-1$ neighbours of $x$, called \textit{successors} of~$x$, 
belong to $\fH_{j-1}$, and one neighbour, called \textit{predecessor} of $x$, lies on~$\fH_{j+1}$.
We denote by $p(x)$ and $s(x)$ the \textit{predecessor} of~$x$ and the set of \textit{successors} of~$x$, respectively.
Notice that $p(x)$ and $s(x)$ depend on the choice of $\om$: in order to simplify the notation, we do not stress this dependence.  

Note that $p\big(p(x)\big)$, also denoted $p^2(x)$, is just 
a vertex in $\fH_{j+2}$.  The~$k^{\mathrm{th}}$ \textit{ancestor} of $x$ is the point $p^k(x) := p\big(p^{k-1}(x)\big)$.  
Sometimes we shall also write $p^0(x)$ in place of $x$.  

Set $s^0(x) := \{x\}$ and $s^1(x) := s(x)$.  If $k\geq 2$, then we define $s^k(x)$ iteratively:  
$$
s^k(x) 
:= \bigcup_{y\in s^{k-1}(x)} \, s(y).
$$   
We shall frequently need to estimate $\bigmod{S_r(x)}$ for various $x\in \fT$.  

\begin{lemma} \label{l: cardinality of spheres}
The following decomposition holds
$$
	S_r(x) 
	= \bigcup_{k=0}^{r} \, \big[S_r(x) \cap \fH_{h(x) + 2k-r}\big].
$$
Furthermore
$S_r(x) \cap \fH_{h(x)-r} = s^{r}(x)$, $S_r(x) \cap \fH_{h(x)+r} = \big\{p^{r}(x)\big\}$ and, if $1\leq k\leq r-1$, then 
$$
S_r(x) \cap \fH_{h(x) + 2k-r} = s^{r-k}\big(p^{k}(x)\big) \setminus s^{r-k-1}\big(p^{k-1}(x)\big).
$$  
Therefore
$$
\bigmod{S_r(x)} 
	= \bigmod{s^r(x)} + \sum_{k=1}^{r-1} \,\sum_{{y\in s(p^k(x))\atop y\neq p^{k-1}(x)}} \, \bigmod{s^{r-k-1}(y)} + 1.   
$$
\end{lemma}

\begin{proof} 
Suppose that $y$ is a point in $S_r(x)$.  Then the intersection of the geodesics $[x,y]$ and $[x,\om)$ is of the form $[x,p^{k}(x)]$ 
for some integer $k$ between~$0$ and $r$.  
Loosely speaking, we can reach $y$ first moving towards $\om$ by $k$ steps along the infinite chain $[x,\om)$, arriving at the point $p^k(x)$, 
and then moving down $r-k$ steps along a geodesic starting at $p^k(x)$.

Therefore $h(y) = h\big(p^k(x)\big) - (r-k) = h(x) + 2k-r$.  This proves the first decomposition.  

Next, observe that the horocycle $\fH_{h(x)-r}$ can be reached in $r$ steps from $x$ only moving down along a geodesic, and 
all the points in $s^r(x)$ can be reached.  Thus, $S_r(x) \cap \fH_{h(x)-r} = s^{r}(x)$, as required.  
Note also that the only point on the horocycle $\fH_{h(x)+r}$ at distance $r$ from $x$ is $p^r(x)$.

Now suppose that $1\leq k\leq r-1$.  
Notice that $S_r(x) \cap \fH_{h(x) + 2k-r}$ is the set of all points in $s^{r-k}\big(p^k(x)\big)$ that are not $(r-k-1)$-successors of $p^{k-1}(x)$.
Therefore if~$v$ belongs to $S_r(x) \cap \fH_{h(x) + 2k-r}$, then it must be a successor of generation $r-k-1$ of a point in
$s\big(p^k(x)\big)\setminus \big\{p^{k-1}(x)\big\}$.
The required formula follows directly from this. 
\end{proof}

In the sequel, $a$ and $b$ will always denote positive integers. 
Suppose that $2\leq a \leq b$.  We say that $\fT$ has \textit{$(a,b)$-bounded geometry} if $a+1 \leq \nu(x) \leq b+1$ for every $x$ in $\fT$,
and the valence function attains the values $a+1$ and $b+1$. 
In the case where $a=b$, we say that $\fT$ is a \textit{$b$-homogeneous tree of degree $b+1$}, and denote it by $\fT_b$ in the sequel; 
each vertex in $\fT_b$ has exactly $b+1$ neighbours.

It is convenient to define the functions $S^b, V^b: \BN \to [1,\infty)$ by 
\begin{equation} \label{f: volume profile}
	V^b(r) 
	:= \begin{cases}
		1                              & \hbox{if $r=0$} \\
		\ds\frac{b^{r+1}+b^r-2}{b-1}   & \hbox{if $r\geq 1$},
	\end{cases}
\quad 
	S^b(r) 
	:= \begin{cases}
		1                 & \hbox{if $r=0$} \\
		(b+1) \, b^{r-1}  & \hbox{if $r\geq 1$}.
	\end{cases}
\end{equation} 
Denote by $B_r^b(x)$ and $S_r^b(x)$ the ball and the sphere with centre $x$ and radius $r$ in $\fT_b$.
Note that $\bigmod{B_r^b(x)} = V^b(r)$ and $\bigmod{S_r^b(x)} = S^b(r)$ for every nonnegative integer~$r$ and for every $x$ in $\fT_b$.

We say that a tree $\fT$ possesses the \textit{Cheeger isoperimetric property} if there exists a positive constant $c$ such that 
$\bigmod{\partial E} \geq c \, \bigmod{E}$ for every subset $E$ of~$\fT$, where $\partial E$ denotes the boundary of $E$, i.e. the set of all
points in $E$ with at least a neighbour in $\fT \setminus E$.
The largest constant $c$ for which this inequality holds for all $E$ is called the \textit{Cheeger isoperimetric constant} of $\fT$,
and is denoted by $c_\fT$ in the sequel.

\begin{proposition} \label{p: Cheeger}
	Suppose that $2\leq a \leq b$ and $\fT$ is a tree in the class $\Upsilon_{a,b}$.  Then $\fT$ possesses the Cheeger isoperimetric property.
\end{proposition}

\begin{proof}
A direct proof can be found in \cite[Lemma~13]{RT}.  
The result is also a straightforward consequence of the analysis performed by R.K.~Wojciechowski in 
his thesis~\cite[Theorem~4.2.2 and Example~4.2.3]{Wo}.
\end{proof}




In this paper the \textit{modified centred maximal operator with parameter $\si>0$}, defined on a tree $\fT$ by 
\begin{equation} \label{f: modified si}
\cM_\si f(x)
:= \sup_{r\in \BN} \, \frac{1}{\mod{B_r(x)}^{1/\si}} \, \int_{B_r(x)} \, \mod{f} \wrt \mu,
\end{equation} 
will play an auxiliary, albeit important, role.  Notice that $\cM_1 = \cM$.  

We refer the reader to \cite[Chapter~V]{SW} for the basic definitions concerning Lorentz spaces.  
Given a positive number $\al$ and a function $f$ on a simple graph $\fG$ (i.e. an undirected graph without self-loops and multi-edges),
we set $E_f(\al) := \bigmod{\{x\in \fG: \bigmod{f(x)} > \al\}}$.
Recall that the $\lorentz{p}{r}{\fG}$ (quasi-) norm may be defined as follows
$$
\bignorm{f}{\lorentz{p}{r}{\fG}} 
= \Big( p \ioty \bigmod{E_f(\al)}^{r/p} \, \al^{r-1} \wrt \al\Big)^{1/r}
$$
if $1\leq r<\infty$, and $\ds\bignorm{f}{\lorentz{p}{\infty}{\fG}} = \sup_{\al>0} \, \al \, \bigmod{E_f(\al)}^{1/p}$.
We shall primarily be concerned with $\lorentz{p}{1}{\fG}$ and $\lorentz{p}{\infty}{\fG}$.
Recall that an operator is \textit{of weak type $(p,p)$} if it is bounded from $\lp{\fG}$
to $\lorentz{p}{\infty}{\fG}$, and, in the case where $1<p<\infty$, that an operator is \textit{of restricted weak type $(p,p)$} if it is bounded from 
$\lorentz{p}{1}{\fG}$ to $\lorentz{p}{\infty}{\fG}$.

\section{Boundedness of $\cM$ on trees with bounded geometry}
\label{s: Boundedness of cM}

It is straightforward to produce examples of trees with the property that~$\cM$ is bounded on $\lp{\fT}$ if and only if $p=\infty$.  
Consider, for instance, the tree $\fT$, characterised by the following property: there is a geodesic ray $\ga := [x_0,x_1,x_2,\ldots]$ in $\fT$ such that 
$\nu(x) = 3$ for every $x\notin \ga$, and $\nu(x_j) = j+3$ for $j=0,1,2,\ldots$.  We consider the horocyclic foliation induced by $\ga$.  Observe that  
\begin{equation*}
\cM\de_{x_j}(x)
=\sup_{r\in\BN} \, \frac{1}{|B_r(x)|}\, \sum_{y\in B_r(x)} \de_{x_j}(y)
=\frac{1}{|B_{d(x,x_j)}(x)|}
\quant  x\in \fT.
\end{equation*}
In particular, if $j\geq 1$ and $x$ is a successor of $x_j$ different from $x_{j-1}$, then 
$\cM\de_{x_j}(x)=|B_1(x)|^{-1} = (\nu(x)+1)^{-1} = 1/4$. Thus, if $p\in [1,\infty)$, then 
\begin{equation*}
	\bignormto{\cM\de_{x_j}}{p}{p}
	\geq \sum_{{x\in s(x_j)}\atop{x\neq x_{j-1}}} \bigmod{\cM\de_{x_j}(x)}^p
	= \frac{j+1}{4^p},
\end{equation*}
which diverges as $j$ tends to infinity.  Since $\norm{\de_{x_j}}{p} = 1$, the operator $\cM$ is unbounded on $\lp{\fT}$.
\textit{A fortiori}, $\cM$ is not of weak type $(1,1)$, for otherwise, interpolating with the trivial bound on $\ly{\fT}$, 
$\cM$ would be bounded on $\lp{\fT}$, $p>1$.

We emphasise that the tree $\fT$ considered in this example has unbounded geometry, in the sense that $\ds\sup_{x\in \fT} \, \nu(x) = \infty$, 
i.e. geodesic balls of radius $1$ can have arbitrarily large mass.  Therefore the counting measure on $\fT$ is not locally doubling.  

In the sequel we shall consider only \textit{locally doubling trees $\fT$ with exponential growth}.  Notice that 
there are examples of such trees in the class $\Upsilon_{a,b}$, for every $a\geq 1$.  
In fact, in this section we shall assume that $a\geq 2$.

We need more terminology.  Fix a reference point $o_b$ in the homogeneous tree $\fT_b$.  For each vertex $y\neq o_b$, consider the set 
\begin{equation*}
E(y)
:= \{ v\in \fT_b: \mod{v} \geq \mod{y} \, \, \hbox{and $y$ lies on the geodesic joining $o_b$ and $v$}\}.
\end{equation*}

\begin{remark}
Observe that a tree $\fT$ in the class $\Upsilon_{a,b}$ is isometric, although not canonically, to a subtree $\wt\fT$ of $\fT_b$.  

To see this, fix reference points $o_b$ in $\fT_b$ and~$o$ in $\fT$.  
Label the neighbours of $o_b$ by $y_1^b,\ldots,y_{b+1}^b$, and the neighbours of $o$ in $\fT$ by $y_1,\ldots,y_{\nu(o)}$.  
If $\nu(o) = b+1$, then we do nothing.  If, instead, $\nu(o) \leq b$, then 
we remove from $\fT_b$ the sets $E(y_j^b)$, $j=\nu(o)+1,\ldots, b+1$.  
The map $\cJ_1$ that associates $o_b$ to $o$, and $y_j^b$ to $y_j$, $j=1,\ldots, \nu(o)$, is an isometry between 
the ball $B_1(o)$ of $\fT$ and the finite subtree $\wt\fT_b^1$ of $\fT_b$ consisting of the vertices
$\{o_b, y_1^b,\ldots,y_{\nu(o)}^b\}$.

We iterate this procedure: assume that we have extended $\cJ_1$ to an isometry $\cJ_k$ between the ball $B_k(o)$ in $\fT$ and a subtree $\wt\fT_b^k$
of $\fT_b$ which contains only vertices at distance at most $k$ from $o_b$.  Suppose that $x$ is one of vertices in $\fT$ at distance $k$ from $o$, 
and let $z_1,\ldots, z_{\nu(x)-1}$ be its neighbours at distance $k+1$ from $o$.  Consider $\cJ_k(x)$, which is a point in $\fT_b$ at distance~$k$ 
from $o_b$, and label its neighbours at distance $k+1$ from $o_b$ by $z_1^b,\ldots, z_{b}^b$.  
Remove from $\fT_b$ the branches $E(z_j^b)$, $j=\nu(x),\ldots, b$.  Then define $\cJ_{k+1} (z_j) = z_j^b$, $j=1,\ldots, \nu(x)-1$.

Repeat this procedure for all points in $S_k(o)$, and denote by 
$\wt\fT_b^{k+1}$ the subtree of $\fT_b$ which contains $\wt\fT_b^k$, and all the points in $\cJ_{k+1}\big(S_{k+1}(o)\big)$.
The map $\cJ_{k+1}$, whose restriction to $B_k(o)$ agrees with $\cJ_k$ and is defined on $S_{k+1}(o)$ as explained above is an
isometry between $B_{k+1}(o)$ and $\wt\fT_b^{k+1}$.   

Finally, define 
$$
\wt\fT 
:= \bigcup_{k=1}^\infty \wt\fT_b^k, 
$$ 
and $\cJ\!\!: \fT \to \wt \fT$ by $\cJ(x) = \cJ_k(x)$ if $x$ is at distance at most $k$ from $o$.  
It is straightforward to check that $\cJ$ is an isometry between $\wt\fT$ and $\fT$.  
\end{remark}

Recall the definition of $\pab$ (see \eqref{f: def tau}).

\begin{theorem} \label{t: main ab}
Suppose that $\fT$ is a tree in $\Upsilon_{a,b}$.  The following hold:
	\begin{enumerate} 
		\item[\itemno1]
			if $a\leq b\leq a^2$, then the operator $\cM$ is of restricted weak type $(\pab,\pab)$ and it is bounded on $\lp{\fT}$  
			for every $p \in (\pab,\infty]$;
			\smallskip
		\item[\itemno2]
			if there exist positive constants $\al$ and $C$ such that $\bigmod{B_r(x)} \geq C \, \al^r$ for all vertices $x$ and nonnegative
			integers $r$, and $b\leq \al^2$, then the operator $\cM$ is of restricted weak type $(p_{\al,b},p_{\al,b})$,
			where $p_{\al,b} := \log_\al b$, and it is bounded on $\lp{\fT}$  
			for every $p$ in $(p_{\al,b},\infty]$.
	\end{enumerate}
\end{theorem}

\begin{proof} 
Since the result for homogeneous trees is well known (see \cite[Theorem~1.5]{NT} and \cite[Theorem~3.1]{CMS1}), we can assume that $a<b$.

First we prove \rmi.  
Clearly $\cM$ is bounded on $\ly{\fT}$, so that the full result follows by interpolating between this and the endpoint result for $\pab$.
Since $\fT$ is isometric to a subtree of $\fT_b$ (see above), we can assume that~$\fT$ is actually a subtree of $\fT_b$.
Hence any function $f$ on $\fT$ can be trivially extended to a function~$f^\sharp$ on $\fT_b$ as follows
$$
	f^\sharp (x) 
	:= f(x) \, \One_\fT (x) 
	\quant x \in \fT_b.
$$
Notice that $\bignorm{f}{\lorentz{p}{q}{\fT}} = \bignorm{f^\sharp}{\lorentz{p}{q}{\fT_b}}$ for all indices $p$ and $q$ such that $1\leq p,q\leq \infty$.

Observe that each vertex $x$ in $\fT$ has, at least, $a+1$ neighbours, so that $\mod{B_r(x)} \geq V^a(r)$ (see \eqref{f: volume profile},
with $a$ in place of $b$, for the definition of $V^a$).  Also notice that 
\begin{equation} \label{f: Cab}
	V^a(r)
	\geq C_{a,b}\,\, V^b(r)^{1/\pab}
	\quant r \in \BN,
\end{equation}
where $\ds C_{a,b} = \Big(\frac{b-1}{b+1}\Big)^{1/\pab}\, \Big(1+\frac{2}{a}\Big)$.  Indeed, the inequality above holds for $r=0$.  
Recalling the formulae for $V^a$ and $V^b$ (see \eqref{f: volume profile}), it suffices to estimate 
\begin{equation} \label{f: intermediate Cab}
\inf_{r\geq 1}\, \frac{(b-1)^{1/{\pab}}}{a-1} \, \frac{a^{r+1}+a^r-2}{(b^{r+1}+b^r-2)^{1/\pab}}. 
\end{equation}
Observe that $b^{1/\pab} = a$ and the function $r \mapsto a+1 - (2/a^r)$ is increasing on $[1,\infty)$.  Therefore
$$
\frac{a^{r+1}+a^r-2}{(b^{r+1}+b^r-2)^{1/\pab}} 
	= \frac{a+1-(2/a^r)}{\big(b+1-(2/b^r)\big)^{1/\pab}}  
	\geq \frac{a+1-(2/a)}{(b+1)^{1/\pab}}. 
$$
By combining this and \eqref{f: intermediate Cab}, we obtain the desired estimate \eqref{f: Cab}.  Hence 
\begin{equation} \label{f: powers Va and Vb}
	\mod{B_r(x)} 
	\geq C_{a,b}\,\, V^b(r)^{1/\pab} 
	\quant r \in \BN.  
\end{equation}
Then 
$$
	\begin{aligned}
	\cM f(x) 
		& =    \sup_{r\in \BN} \, \frac{1}{\mod{B_r(x)}} \, \int_{B_r(x)} \, \mod{f(y)} \wrt \mu(y)\\
		& \leq C_{a,b}^{-1} \, \sup_{r\in \BN} \, \frac{1}{V^b(r)^{1/\pab}} \, \int_{B_r^b(x)} \, \mod{f^\sharp(y)} \wrt \mu(y);
	\end{aligned}
$$
the last inequality follows from \eqref{f: powers Va and Vb} and the trivial fact that  
$\ds \int_{B_r(x)} \, \mod{f(y)} \wrt \mu(y) = \int_{B_r^b(x)} \, \mod{f^\sharp(y)} \wrt \mu(y)$ (for $f^\sharp$ vanishes on $\fT_b\setminus \fT$).  
Denote by $\cM_\si^b$ the modified centred maximal function (with parameter $\si$) on $\fT_b$.  Altogether, we have proved that 
\begin{equation} \label{f: centred M dominated modified M}
	\cM f (x) 	 
	\leq C_{a,b}^{-1} \,\, \cM_{\pab}^b f^\sharp (x) 
	\quant x \in \fT.   
\end{equation}

If $b=a^2$, then $\pab = 2$, and
$$
	\begin{aligned}
	C_{a,b} \bignorm{\cM f}{\lorentz{2}{\infty}{\fT}}
		& \leq  \bignorm{\cM_{2}^b f^\sharp}{\lorentz{2}{\infty}{\fT}} \\
		& \leq  \bignorm{\cM_{2}^b f^\sharp}{\lorentz{2}{\infty}{\fT_b}} \\
		& \leq  \bigopnorm{\cM_{2}^b}{\lorentz{2}{1}{\fT_b};\lorentz{2}{\infty}{\fT_b}} \, \bignorm{f^\sharp}{\lorentz{2}{1}{\fT_b}} \\
		& =     \bigopnorm{\cM_{2}^b}{\lorentz{2}{1}{\fT_b};\lorentz{2}{\infty}{\fT_b}} \, \bignorm{f}{\lorentz{2}{1}{\fT}} 
	\end{aligned}
$$
for every $f$ in $\lorentz{2}{1}{\fT}$, as required.  
Note that the second and the fourth inequality above are trivial; the third follows from Veca's result \cite[Theorem~5.1]{V}.  

Assume now that $b<a^2$, i.e. $\pab \in (1,2)$.  Denote by $o_b$ a fixed, but otherwise arbitrary, reference point in $\fT_b$.
For each nonnegative integer~$r$ consider the functions $\Phi_r$ and $\Phi$ on~$\fT_b$, defined by
$$
	\Phi_r := \frac{\One_{B_r^b(o_b)}}{V^b(r)^{1/\pab}},
	\qquad\qquad
	\Phi :=\sup_{r\in \BN} \, \Phi_r. 
$$
Clearly $\Phi (x) = V^b(\mod{x})^{-1/\pab}$, where $\mod{x}$ denotes the distance from $x$ to $o_b$ in~$\fT_b$.  

Observe that $\Phi$ belongs to $\lorentz{\pab}{\infty}{\fT_b}$.  
Indeed, $E_\la := \big\{x\in \fT_b: \Phi(x) >\la\big\}$ is empty when $\la \geq 1$, and it is equal to the ball $B_{m(\la)}^b (o_b)$, where $m(\la)$
is the largest integer $r$ such that $V^b(r) <\la^{-\pab}$.  Clearly 
$$
\bigmod{B_{m(\la)}^b (o_b)}
< \la^{-\pab}
\quant \la >0,
$$
so that $\bignorm{\Phi}{\lorentz{\pab}{\infty}{\fT_b}} \leq 1$.   Since 
$$
\cM_{\pab}^b f^\sharp (x) 
= \sup_{r\in \BN} \, \int_{\fT_b} \mod{f^\sharp(y)} \,\frac{ \One_{B_r^b(x)}(y)}{V^b(r)^{1/\pab}} \wrt \mu(y) 
\leq  \, \mod{f^\sharp}*\big(\sup_{r\in \BN} \Phi_r\big)(x)
=  \mod{f^\sharp}*\Phi(x),
$$
where $*$ denotes the convolution on $\fT_b$ (see \cite[formula (2.5)]{CMS2}), the sharp form of the 
Kunze--Stein theorem on the group of automorphisms of~$\fT_b$ (see \cite[Theorem 1]{CMS2}) implies that 
$$
\bignorm{\mod{f^\sharp}*\Phi}{\lorentz{\pab}{\infty}{\fT_b}} 
\leq C_p\, \bignorm{\Phi}{\lorentz{\pab}{\infty}{\fT_b}} \bignorm{f^\sharp}{\lorentz{\pab}{1}{\fT_b}},
$$
where $C_p$ is independent of $f^\sharp$.  By combining the last two estimates and \eqref{f: centred M dominated modified M}, we obtain 
$$
\begin{aligned}
\bignorm{\cM f}{\lorentz{\pab}{\infty}{\fT}}
	& \leq C_{a,b}^{-1}  \bignorm{\cM_{\pab}^b f^\sharp}{\lorentz{\pab}{\infty}{\fT}} \\
	& \leq C_{a,b}^{-1}  \bignorm{\cM_{\pab}^b f^\sharp}{\lorentz{\pab}{\infty}{\fT_b}} \\
	& \leq C_p \, C_{a,b}^{-1}  \bignorm{f^\sharp}{\lorentz{\pab}{1}{\fT_b}} \\
	& =    C_p \, C_{a,b}^{-1}  \bignorm{f}{\lorentz{\pab}{1}{\fT}}, 
\end{aligned}
$$
as required.  

Next we indicate the modifications to the proof of \rmi\ needed to prove \rmii.  We use the same notation as in \rmi.
The main observation is that the lower estimate for $\mod{B_r(x)}$ implies the upper estimate 
$$
\frac{1}{|B_r(x)|} \int_{B_r(x)} |f| \wrt\mu 
\leq  \frac{C}{\al^r} \int_{B_r(x)} |f| \wrt\mu 
=     \frac{C}{b^{r/p_{\al,b}}} \int_{{B}_r^b(x)} |f^\sharp| \wrt\mu;
$$
the equality above follows from the trivial fact that $b^{1/p_{\al,b}} = \al$.  Now, $\ds b^r \geq \frac{1}{3}\, V^b(r)$, whence the 
following pointwise bound holds
$$
\cM {f} (x)
\leq C \, \cM_{p_{\al,b}}^b f^\sharp (x) 
\quant x \in \fT,
$$
and we can argue, \textit{mutatis mutandis}, as in the proof of \rmi.  
\end{proof}

Theorem~\ref{t: main ab}~\rmi\ above is sharp, as shown in Proposition~\ref{p: counterexample Stromberg} below.  

Preliminarily, we need to introduce the tree $\fS_{a,b}$ that will play an important role in what follows:
$\fS_{a,b}$  may be thought of as a discrete counterpart of the Riemannian manifold considered by Str\" omberg \cite[Remark~3]{Str}.  
The main idea is to allow only two possible valences, namely $a+1$ and $b+1$, and to keep the points with valence
$a+1$ \textit{suitably} \lq\lq well separated'' from those with valence $b+1$.  Specifically, we require that $\nu(x) = b+1$ for all points~$x$
belonging to the horocycles $\fH_j$, $j\geq 1$, and $\nu(x) = a+1$ for all the other vertices $x$.
We denote by $\fS_a$ and $\fS_b$ the sets of all vertices in $\fS_{a,b}$ with valence $a+1$ and $b+1$, respectively.

\begin{proposition} \label{p: counterexample Stromberg}
The following hold:
\begin{enumerate}
	\item[\itemno1]
		if $a< b \leq a^2$ and $p \in [1,\pab)$, then $\cM$ is not of restricted weak type $(p,p)$ on $\fS_{a,b}$ (hence $\cM$ is unbounded 
		on $\lp{\fS_{a,b}}$);  
	\item[\itemno2]
		if $b> a^2$, then $\cM$ is not of restricted weak type $(p,p)$ on $\fS_{a,b}$ for every $p \in (1,\infty)$; hence $\cM$ is
		bounded on $\lp{\fS_{a,b}}$ if and only if $p=\infty$.
\end{enumerate}
\end{proposition}

\begin{proof}
Consider the geodesic ray $\om = \{o,p(o),p^2(o),p^3(o), \ldots\}$.

First we prove \rmi.
For large even positive integers $n$, set
$$
E_n
:= s^{n}\big(p^n(o)\big).  
$$
Observe that $E_n$ is contained in $\fH_0$, and that $|E_n| = b^n$.  We estimate $\cM \de_{p^n(o)}(x)$ for all $x$ in $E_n$.  

Suppose that $x$ belongs to $E_n$.  Then $\bigmod{s^n(x)} = a^n$, because $x$ and all its successors lie in $\fS_a$.  
Furthermore $p^k(x)$, $k=1,\ldots,n$, lies in $\fS_b$ together with all of its successors up to the step $k-1$, whereas its successors of steps
ranging from $k$ up to $n$ lie in $\fS_a$.  Therefore if $1\leq k\leq n/2$, then 
\begin{equation} \label{f: cardinality of speres Sab}
	\bigmod{S_n(x) \cap \fH_{h(x) + 2k-n}}
	= (b-1) \,b^{k-1} \, a^{n-2k}.  
\end{equation}
Now we distinguish two cases, according to whether $b<a^2$ or $b=a^2$.  In the first case, by Lemma~\ref{l: cardinality of spheres},
$$
\begin{aligned}
	\bigmod{S_n(x)}
	& = a^n + \sum_{k=1}^{n/2} \, (b-1) \,b^{k-1} \, a^{n-2k} + (b-1) \, \sum_{k=n/2+1}^{n-1} \, b^{n-k-1} + 1\\
	& = a^n \Big[1+ (b-1)\, \frac{1-(b/a^2)^{n/2}}{a^2-b}\Big] + b^{n/2-1}\\
	& \leq \be_{a,b}\, a^n,
\end{aligned}
$$
where $\ds \beta_{a,b} :=  \frac{1}{a^2}\, \frac{a^4-b}{a^2-b}$.  Thus 
$$
	\cM \de_{p^n(o)} (x) 
	\geq \frac{1}{\mod{B_n(x)}} 
	\geq  \frac{c_{\fS_{a,b}}}{\mod{S_n(x)}}  
	\geq  \frac{c_{\fS_{a,b}}}{{\be_{a, b}}} \, a^{-n}
	\quant x \in E_n, 
$$
where $c_{\fS_{a,b}}$ denotes the Cheeger constant of $\fS_{a,b}$.  Set $\ds\la_n := \frac{c_{\fS_{a,b}}}{{\be_{a, b}}} \, a^{-n}$.  
If $\cM$ were bounded from $\lorentz{p}{1}{\fS_{a,b}}$ to $\lorentz{p}{\infty}{\fS_{a,b}}$
for some $p\in [1,\pab)$, then there would exist a constant $C$, independent of~$n$, such that 
$$
	\bigmod{\{y\in \fS_{a,b}: \cM \de_{p^n(o)} (y) > \la_n \}}
	\leq C\, \la_n^{-p}.
$$
Since the level set $\{\cM \de_{p^n(o)} > \la_n \}$ contains $E_n$, and $\bigmod{E_n} = b^n$, we should have $b^n \leq C\, \la_n^{-p}$, which, 
however, fails for $n$ large as soon as $p< \pab$.

By arguing similarly, we see that in the second case, i.e. when $b=a^2$, $\ds \bigmod{S_n(x)} \leq n\, a^n$ 
for all large $n$.  By choosing $\la_n' := c_{\fS_{a,b}}  \,n^{-1} \, a^{-n}$, and following the steps above, we can still conclude that $p\geq  \pab$.

Next we prove \rmii.  The idea of the proof is similar to that of \rmi, but with an important variant.  For large positive integers $m$ and $n$, 
let $E_n$ be as in~\rmi, and set
$$
F_{n,m}
	:= s^{n+m}(p^n(o)).
$$
It is straightforward to check that $\bigmod{F_{n,m}} = b^n \, a^m$.  We shall estimate $\cM\One_{E_n}$ on $F_{n,m}$, and choose $m$ 
suitably as a function of $n$.  

We need to estimate $\bigmod{S_{m+2n}(x)}$ for every $x$ in $F_{n,m}$.  For the sake of simplicity, we write provisionally
$h$ in place of $m+2n$.  Notice that  
$$
\begin{aligned}
|S_{h}(x)|
&  = a^{h}+(a-1)\sum_{j=1}^{m}a^{h-j-1}+(b-1)\sum_{j=m+1}^{m+n-1}b^{j-m-1}a^{h-2j+m}\\ 
&  \qquad+(b-1)\sum_{j=m+n}^{h-1}b^{h-j-1}+1.
\end{aligned}
$$
Observe that $\ds (b-1)\, \sum_{j=m+1}^{m+n-1}b^{j-m-1}a^{m+2n-2j+m} \leq  b^{n+1}$, because $b>a^2$.  Simple estimates yield the bound 
$$
\bigmod{S_{h}(x)}
\leq 2a^{h} + 2b^{n+1}.  
$$
Recall that $\pab > 2$, for $b>a^2$.   
Set $m_0 := \lfloor n (\pab-2)\rfloor$ and $h_0=m_0+2n$, and note that $m_0$ is the biggest nonnegative integer such that $a^{h_0} \leq b^n$.  Then,
by the Cheger isoperimetric inequality, 
$$
\bigmod{B_{h_0}(x)}
\leq c_{\fS_{a,b}}^{-1} \, \bigmod{S_{h_0}(x)}
\leq 2\big(1+b\big) \,c_{\fS_{a,b}}^{-1} \, b^n.   
$$
Therefore, if $x \in F_{n,m}$, then 
$$
\cM \One_{E_n}(x) 
\geq \frac{1}{|B_{h_0}(x)|} \int_{B_{h_0}(x)} \!\!\One_{E_n} \wrt \mu 
=    \frac{\bigmod{E_n}}{\bigmod{B_{h_0}(x)}} 
\geq \frac{c_{\fS_{a,b}}}{2(1+b)},
$$
and $F_{n,m_0} \subseteq \big\{\cM \One_{E_n} \geq \la_{a,b}\big\}$, where $\la_{a,b} := c_{\fS_{a,b}}/[2(1+b)]$.  
If $\cM$ were bounded from $\lorentz{p}{1}{\fS_{a,b}}$ to $\lorentz{p}{\infty}{\fS_{a,b}}$ for some $p \in [1,\infty)$, then there would exist
a constant $C$, independent of $n$, such that 
$$
a^{m_0}\, b^n 
= \bigmod{F_{n,m_0}} 
\leq \bigmod{\big\{\cM \One_{E_n} \geq \la_{a,b}\big\}}  
\leq \frac{C}{\la_{a,b}^p} \, \bigmod{E_n}
= \frac{C}{\la_{a,b}^p} \, b^n.  
$$
This inequality fails for $n$ large, because $a^{m_0}$ tends to infinity as $n$ tends to infinity.  

This concludes the proof of \rmii, and of the proposition.  
\end{proof}

%
%

Our next goal is to produce examples of trees $\fT$ in $\Upsilon_{a,b}$ with the property that $\cM$ is bounded on $\lp{\fT}$ if and only if $p>q$, 
where $q$ is ``any'' rational number in $(1,\pab)$.  
This will be pursued in the next proposition.  First, we need a few preliminaries.  

For positive integers $m$ and $n$, consider the tree $\fU_{a,b}^{m,n}$ in $\Upsilon_{a,b}$ characterised by the property that $\nu(x) = b+1$ 
if either $h(x) \geq 1$, or $-k(m+n)-n\leq h(x)< -k(m+n)$, $k=1,2,\ldots$, and $\nu(x) = a+1$ otherwise.  Define
\begin{equation} \label{f: def al}
\al:=(a^mb^n)^{1/(m+n)}.   
\end{equation}
We \textit{claim} that 
\begin{equation} \label{f: claim lb}
\bigmod{B_r(x)} 
\geq \al^{r-m-n}
\quant r \in \BN \quant x\in \fU_{a,b}^{m,n}.  
\end{equation}
Clearly, given a nonnegative integer $r$, the balls of radius $r$ with smallest volume are contained in $\{y\in \fU_{a,b}^{m,n}: h(y) \leq 0\}$.  
Furthermore, if $B_r(x)$ is such a ball, then $\bigmod{B_r(x)}\geq \bigmod{s^r(x)}$.  If $k$ is a positive integer such that
$k(m+n) \leq r<(k+1)\, (m+n)$, then 
$$
\bigmod{s^r(x)} 
\geq (a^mb^n)^k
= (a^mb^n)^{\lfloor r/(m+n)\rfloor}
\geq \al^{r-m-n}.   
$$
If, instead, $0\leq r < m+n$, then 
$$
|B_r(x)| 
\geq 1  
\geq \al^{r-m-n}.
$$  
The claim follows by combining these estimates.

Suppose now that $\al\leq b\leq \al^2$, or equivalently, adopting the notation introduced in Theorem~\ref{t: main ab}~\rmii,
$p_{\al,b} := \log_\al b \leq 2$.  Then we can apply Theorem~\ref{t: main ab}~\rmii, and
conclude that $\cM$ is of restricted weak type $(p_{\al,b},p_{\al,b})$, and it is bounded on $\lp{\fU_{a,b}^{m,n}}$  
for every $p$ in $(p_{\al,b},\infty]$.  

Notice that if $\pab \leq 2$, i.e. if $b\leq a^2$, then the condition $b\leq \al^2$ is satisfied for every pair $m, n$ of positive integers. 

\begin{proposition} \label{p: denseness}
Suppose that $a$ and $b$ are integers such that $2\leq a< b \leq a^2$ and that $q$ is in $(1,\pab)$.   Then for every $\vep >0$ there exists $s$ such that
$1<q-\vep< s\leq q$ and a tree $\fT$ in $\Upsilon_{a,b}$ such that $\cM$ is bounded on $\lp{\fT}$ if and only if $p>s$.
\end{proposition}

\begin{proof}
For each pair of positive integers $m$ and $n$, consider the number $\al$ defined in \eqref{f: def al}.  The tree $\fU_{a,b}^{m,n}$ is in the class 
$\Upsilon_{a,b}$ and its geodesic balls satisfy the lower estimate \eqref{f: claim lb}.  Notice that $a\leq \al$, whence $\log_\al b \leq  \log_ab \leq 2$; 
the last inequality follows from the assumption $b\leq a^2$.  

By Theorem~\ref{t: main ab}~\rmii, the maximal function $\cM$ is bounded on $\lp{\fU_{a,b}^{m,n}}$ for $p>p_{\al,b} := \log_\al b$ and 
it is of restricted weak type $(p_{\al,b},p_{\al,b})$.  

Notice that any rational number in the interval $(0,1)$ can be written in the form $m/(m+n)$, for suitable positive integers $m$ and $n$.  Now,
$$
\log_\al b 
= \frac{\log b}{\log \al} 
= \frac{\log b}{\frac{m}{m+n} \, \log a + \frac{n}{m+n} \, \log b}.
$$
The denominator of the latter ratio is a convex combination of $\log a$ and $\log b$ with rational coefficients.  As $m$ and $n$ vary,
the denominator is dense in the interval $(\log a, \log b)$, so that the set of all numbers of the form $\log_\al b$ is dense in $(1,\pab)$.
Hence for every $\vep>0$ we can choose $m$ and $n$ so that $p_{\al,b} \in (q-\vep,q]$, and define $s:= p_{\al,b}$. 

It remains to prove that $\cM$ is unbounded on $\lp{\fU_{a,b}^{m,n}}$ for $p<s$.  We follow the strategy of the proof of 
Proposition~\ref{p: counterexample Stromberg}~\rmi.

For large positive integers $k$, set $r := 2k(m+n)$.  Fix a reference point $o$ in $\fH_0$ and consider
$$
E_r
:= s^{r}\big(p^r(o)\big).  
$$
Observe that $E_r$ is contained in $\fH_0$, and that $|E_r| = b^r$.  We estimate $\cM \de_{p^r(o)}(x)$ for all $x$ in $E_r$.  

By arguing much as in the proof of Proposition~\ref{p: counterexample Stromberg}~\rmi, it is not hard to prove that there exists a constant $C$
such that 
$$
	\bigmod{B_r(x)}
	\leq 
	\begin{cases}
		C\al^r 		& \hbox{if $s<2$} \\
		Cr\al^r	 	& \hbox{if $s=2$}.
	\end{cases}
$$
Suppose that $s<2$.  If $x$ belongs to $E_n$, then
$$
	\cM \de_{p^r(o)} (x) 
	\geq \frac{1}{\mod{B_r(x)}} 
	\geq  C\, \al^{-r}.  
$$
Set $\ds\la_r := C\, \al^{-r}$.  
If $\cM$ were bounded from $\lorentz{p}{1}{\fU_{a,b}^{m,n}}$ to $\lorentz{p}{\infty}{\fU_{a,b}^{m,n}}$
for some $p\in [1,s)$, then there would exist a constant $C$, independent of~$r$, such that 
$$
	\bigmod{\{y\in \fU_{a,b}^{m,n}: \cM \de_{p^n(o)} (y) > \la_r \}}
	\leq C\, \la_r^{-p}.
$$
Since the level set $\{\cM \de_{p^r(o)} > \la_r \}$ contains $E_r$, and $\bigmod{E_r} = b^r$, we should have $b^r \leq C\, \la_r^{-p}$, equivalently
$b^{r} \leq C\, \al^{pr}$.   This inequality fails for $n$ large as soon as $p<s$.

The case where $s=2$ can be treated similarly.  We omit the details.
\end{proof}

We now complement Theorem~\ref{t: main ab}~\rmii, by proving a simple sufficient condition, concerning the volume growth of geodesic balls
of trees in the class~$\Upsilon_{a,b}$, that guarantee that $\cM$ is of weak type $(1,1)$.   Our result, is, in fact, a rather direct consequence of the
method developed in \cite{NT} (see, in particular, Lemma~5.1 and the proof of Theorem~1.5, and its generalisation in \cite[Theorem~4.1]{ST}).

For the sake of completeness, we include a complete proof of the next lemma.  

\begin{lemma} \label{l: TAO}
Suppose that $\fT$ is in $\Upsilon_{a,b}$.  Assume that $\al$ and $c_2$ are constants such that 
$$
|S_r(x)| \le c_2 \, \al^r 
\quant x \in \fT \quant r \in \BN.
$$
Then
$$
\bigmod{\{ (x,y) \in A \times B: d(x,y)=r\}}
\le 8c_2\, \sqrt{|A| \,|B|\, \al^{r}}
$$
for all finite subsets $A$ and $B$ of $\fT$ and every $r$ in $\BN$.  
\end{lemma}

\begin{proof} 
Think of $\fT$ as a rooted tree at $o$.  

Suppose that $\mod{x} = j$ and denote by $z_0 := o, \ldots, z_j :=x$ the points on the geodesic joining $o$ and $x$.  For each point $y$ in $S_r(x) \cap S_k(o)$, 
denote by $x \wedge y$ the vertex with biggest distance from $o$ that belongs to both geodesics $[o,x]$ and $[o,y]$, and set $m := d(x,x \wedge y)$. 
Then $x \wedge y = z_{j-m}$, and 
$$
r= d(x,y) = d(x,z_{j-m}) + d(z_{j-m}, y)=m + k-(j-m) = 2m+k-j.
$$  
Thus, given $r$, $j$ and $k$, the integer $m$ is uniquely determined, viz. $m = (r+j-k)/2$.  Consequently if $\mod{x} = j$, then all the points in $S_r(x) \cap S_k(o)$
belong to the sphere with centre $z_{j-m}$ and radius $r-m$.  Hence, if $\mod{x} = j$, then 
\begin{equation} \label{f: product I}
\bigmod{\{y \in S_r(x): |y|=k\}}
\le c_2\, \al^{r-m}.
\end{equation}
A similar reasoning yields that if $\mod{y} = k$, then 
\begin{equation} \label{f: product II}
\bigmod{\{x \in S_r(y): |x|=j\}}
\le c_2\, \al^{m}.
\end{equation}
For every pair $(j,k)$ of nonnegative integers define 
$$
A_j:=A \cap S_j(o),
\qquad
B_k:=B \cap S_k(o),
$$
and 
$$
E_{j,k,r}
:= \{ (x,y) \in A_j \times B_k: d(x,y)=r\}.  
$$ 
Furthermore, set 
$
E_{r}
:= \{ (x,y) \in A \times B: d(x,y)=r\}.  
$
Observe that 
$$
\bigmod{E_{j,k,r}}
=    \sum_{x \in A_j } \sum_{y \in B_k} \One_{S_r(x)}(y) 
=    \sum_{y \in B_k } \sum_{x \in A_j} \One_{S_r(y)}(x).  
$$
Now \eqref{f: product I} and \eqref{f: product II} imply that
$$
	\bigmod{E_{j,k,r}}
	\le  c_2\min\, \big(|A_j| \, \al^{r-m},|B_k|\, \al^m\big).
$$
Clearly
$$
\begin{aligned}
\bigmod{E_r}
	& =   \sum_{j,k=0}^\infty \, \bigmod{E_{j,k,r}} \\
	& =   \sum_{m=0}^r\, \sum_{j,k: \, k=j+r-2m}\, \bigmod{E_{j,k,r}} \\
	& \le c_2 \sum_{m=0}^r \,  \sum_{j,k: \, k=j+r-2m} \!\! \min \big( |A_j| \,\al^{r-m},|B_k|\, \al^m \big).
\end{aligned}
$$
Define $e_j:=|A_j|/\al^j$ and $d_k:=|B_k|/\al^k$, and assume that $k=j+r-2m$.  Then   
$$
\begin{aligned}
\min \big(|A_j|\, \al^{r-m},|B_k|\, \al^m \big)
	& = \min \big(\al^{r-m+j}e_j, \al^{m+k}d_k \big) \\
	& = \al^{(k+j+r)/2} \, \min\big(e_j,d_k\big);
\end{aligned}
$$
the last equality follows from the fact that $r-m+j = m+k$.
Altogether, we see that 
$$
|E_r|  
\le c_2\, \al^{r/2}\sum_{j,k =0}^\infty \, \al^{(k+j)/2}\, \min (e_k,d_j).
$$
Henceforth, we proceed exactly as in the last part of the proof of \cite[p.~759--760]{NT}.  We include the details for the reader's convenience. 

We claim that 
$$
\sum_{j,k =0}^\infty \al^{(k+j)/2}\, \min \big(e_j,d_k) 
\le 8 \, \sqrt{|A|\,|B|}.
$$
Indeed, for some $\beta$ to be chosen later,
$$
\begin{aligned}
	\sum_{j,k =0}^\infty \al^{(j+k)/2}\min(e_j,d_k)
& \leq   \sum_{j,k: \, j<k+\beta} \al^{(j+k)/2}d_k  +  \sum_{j,k: \, j \ge k+\beta} \al^{(j+k)/2}e_j \\ 
& =      \sum_{k=0}^\infty d_k \al^{k/2} \sum_{j < k+\beta} \al^{j/2}+\sum_{j=0}^\infty \al^{j/2}e_j \sum_{k \le j-\beta}\al^{k/2} \\ 
& \leq   4 \, \Big(\al^{\beta/2} \sum_{k=0}^\infty d_k \al^{k} +\al^{-\beta/2}\sum_{j=0}^\infty \al^{j}e_j\Big)\\ 
& =      4 \, \Big(\al^{\beta/2} |B| +\al^{-\beta/2}|A|\Big).
\end{aligned} 
$$
Optimizing in $\beta$, we obtain that $\beta$ such that $\al^{\beta/2}=\sqrt{|A|/|B|}$ is the best constant. This gives the desired conclusion.
\end{proof}


\begin{theorem} \label{t: weak11}
Suppose that $\fT$ is in $\Upsilon_{a,b}$.  Assume that there exist positive constants $\al$, $c_1$ and $c_2$ such that 
\begin{equation}\label{f: equicomparability}
c_1\, \al^r \le |S_r(x)| \le c_2 \, \al^r 
\quant x \in \fT, \quant r \in \BN.
\end{equation}
Then $\cM$ is of weak type (1,1).
\end{theorem}

\begin{proof}
Set 
$$
\cE_\fT(r) 
:= \sup_{A,B \subset \mathfrak{T}: 0<|A|,|B|<\infty} \frac{1}{|A||B|} \Big(\sum_{ x \in B}\frac{|A \cap S_r(x)|}{|S_r(x)|}\Big)^2,
$$ 
and observe that Lemma~\ref{l: TAO} implies that 
\begin{equation} \label{f: claim}
\cE_\fT(r) 
	\leq 64\Big(\frac{c_2}{c_1}\Big)^2\, \al^{-r}.
\end{equation} 
Then, by \cite[Theorem 4.1]{ST}, there exists a constant $C$ such that 
\begin{align*}
	\bigopnorm{\cM^o}{\lu{\fT}; \lorentz{1}{\infty}{\fT}}
        & \leq C\, \sup_{n \in \BN} \, 2^{n/2} \, \sum_{r \in \BN: \al^r \ge 2^{n-1}/c_2} \cE_\fT(r) \, c_2^{1/2} \al^{r/2} \\ 
	& \leq C\, \sup_{n \in \BN} \, 2^{n/2} \, \sum_{r \in \BN: \al^r \ge 2^{n-1}/c_2} \, \al^{-r/2} \\ 
	& \leq C\, \sup_{n \in \BN} \, 2^{n/2} \, 2^{-n/2},  
\end{align*}
which is finite.  Here $\cM^o$ is the spherical maximal operator, which is pointwise comparable to $\cM$, because every ball can be written 
as a disjoint union of spheres.
\end{proof}

As a consequence, we show that for every pair $(a,b)$ of integers, with $2\leq a <b$, there are trees in the class $\Upsilon_{a,b}$, 
where $\cM$ is of weak type $(1,1)$.  
The main idea is to arrange matters so that vertices with different valences are not too neatly separated.  

A simple example is the \textit{semi-homogeneous tree} $\fV_{a,b}$ with valences $a$ and~$b$, i.e. a tree in which every vertex has either valence $a+1$
or $b+1$, and such that adjacent vertices have different valences.

Indeed, it is straightforward to check that there exist positive constants $c_1$ and $c_2$, depending on $a$ and~$b$, such that 
\begin{equation} \label{f: growth semihomogeneous}
c_1\, (ab)^{r/2}
\leq \bigmod{B_r(x)} 
\leq c_2\, (ab)^{r/2}
\quant x \in \fS \quant r \in \BN.
\end{equation}
Theorem~\ref{t: weak11} then implies that $\cM$ can be of weak type $(1,1)$, no matter how large the ratio $b/a$ is.

\section{The uncentred maximal function}
\label{s: Unboundedness nc}

Clearly $\cN$ dominates $\cM$ pointwise.  However, contrary to what happens on doubling metric measured spaces, $\cN$ may be much larger than $\cM$.
For instance, $\cM$ is bounded on $\lp{\fT_b}$ for all $p>1$, whereas $\cN$ is bounded on $\lp{\fT_b}$ only for $p>2$.

Furthermore, if $\fT$ belongs to the class $\Upsilon_{a,b}$, then
it is straightforward to check that there exists a constant~$C$ (depending only on $a$ and $b$) such that
$$
\cN f(x)
\leq C\, \cM_{2\pab} f (x)
\quant x \in \fT:
$$
here $\pab$ is as in \eqref{f: def tau}.
This follows directly from the remark that
if $r$ is a nonnegative integer and $B$ is a ball of radius $r$ containing the point~$x$, then the ball $B_{2r}(x)$ contains $B$, 
and $\mod{B} \geq C \, \mod{B_{2r}(x)}^{1/(2\pab)}$.  

Thus, $L^p$ bounds for the modified maximal operator $\cM_{2\pab}$ imply similar bounds for $\cN$.   In particular, 
$\cM_2$ is of restricted weak type $2$ \cite[Theorem~5.1]{V} on the homogeneous tree $\fT_b$, whence so is $\cN$ (for in this case $\pab = 1$). 
By interpolation $\cN$ is bounded on $\lp{\fT_b}$ for all $p > 2$.  

It is natural to wonder what happens in the case where $a<b$.  It would not be unreasonable to conjecture that for each tree $\fT$ in the 
class $\Upsilon_{a,b}$ with $b\leq a^2$
there exists a threshold $q<\infty$ such that $\cN$ is bounded on $\lp{\fT}$ for all $p>q$.  However, this fails, as we show in 
Theorem~\ref{t: unboundedness nc} below: there exist trees in the class $\Upsilon_{a,b}$ such that $\cN$ is bounded on $\lp{\fT}$ if and only if $p=\infty$.  
Furthermore Theorem~\ref{t: unboundedness nc} points out the striking difference that can occur between the $L^p$ boundedness properties of the centred and
uncentred maximal functions on graphs with exponential volume growth.  As already mentioned, this difference occurs also in the case of homogeneous 
trees ($\cM$ is of weak type $(1,1)$, $\cN$ is of restricted weak type $2$ and unbounded on $\lp{\fT_b}$ for $p<2$), but it can be much wider
for nonhomogeneous trees.

Suppose that $a$ and $b$ are integers such that $2\leq a < b$.  Preliminarily, we define a tree $\fF_{a,b}$ in the class $\Upsilon_{a,b}$ that
will play a fundamental role in this section.  
The tree $\fF_{a,b}$ contains a reference point $o$ with $a+1$ neighbours $g_0,\ldots,g_a$
with the following property: the vertex $g_0$ and all the vertices of the tree on geodesics starting with $[o,g_0]$ have valence $b+1$,
and all the vertices $g_1, \ldots, g_a$ together with the vertices on geodesics starting with $[o,g_j]$ for some $j \in \{1,\ldots,a\}$
have valence $a+1$.  Denote by $\fF_a$ and $\fF_b$ the set of vertices in $\fF_{a,b}$ with valence $a+1$ and $b+1$, respectively.  
Clearly $\fF_{a,b} = \fF_a \cup \fF_b$.   

\begin{theorem} \label{t: unboundedness nc}
Suppose that $2\leq a <b$.  The following hold:
	\begin{enumerate}
		\item[\itemno1]
			$\cN$ is unbounded on $\lp{\fF_{a,b}}$ for all $p$ in $[1,\infty)$; 
		\item[\itemno2]
			$\cM$ is bounded on $\lp{\fF_{a,b}}$ whenever $p\in (1,\infty)$, and it is of weak type $(1,1)$. 
	\end{enumerate}
\end{theorem}

\begin{proof}
Recall that $\pab$ denotes $\log_ab$, and observe that $\pab >1$.  

First we prove \rmi.  For large integers $n$, consider the sets
$$
	E_n := \big\{x \in \fF_b: \mod{x} = n\big\}
	\quad\hbox{and}\quad
	F_n := \big\{x \in \fF_a: \mod{x} = \lfloor (2\pab - 1)\, n\rfloor\big\},
$$
where $\lfloor \cdot \rfloor$ denotes the integer part function.  Clearly 
$$
	\bigmod{E_n} = b^{n-1}\quad\hbox{and}\quad \bigmod{F_n} = a^{\lfloor (2\pab - 1)\, n\rfloor}.   
$$
We want to estimate from below the measure of the level set $\{\cN \One_{E_n}>c\}$, where $c$ is a suitable positive constant, independent
of $n$, that will be determined below. 

For $x$ in $F_n$, consider the unique point $y$ on the geodesic joining $x$ and~$o$ at distance $\lfloor \pab \, n\rfloor$ from $x$.  
Of course $y$ depends on $x$, but this has no consequence for what follows.  Observe that 
$$
	d(o,y) 
	= \lfloor (2\pab - 1)\, n\rfloor-\lfloor \pab \, n\rfloor 
	\leq (2\pab - 1)\, n- \pab \, n + 1
	=  \pab \, n- n + 1, 
$$  
so that 
$$
	d(y, E_n) 
	=    d(y,o) + d(o,E_n) 
	=    \lfloor (2\pab - 1)\, n\rfloor-\lfloor \pab \, n\rfloor + n 
	\leq \pab \,n + 1.  
$$
Set $R_n :=\lfloor \pab \, n\rfloor + 2$.  The estimates above imply that~$B_{R_n}(y)$ contains both $x$ and $E_n$.  
Therefore, given $x$ in $F_n$, there exists~$y$ such that 
$$
	\cN \One_{E_n} (x) 
	\geq \frac{1}{\bigmod{B_{R_n}(y)}} \, \int_{B_{R_n}(y)} \, \One_{E_n} \wrt \mu  
	= \frac{\bigmod{E_n}}{\bigmod{B_{R_n}(y)}}.
$$
By the Cheeger isoperimetric property (see Proposition~\ref{p: Cheeger}), ${\bigmod{B_R(y)}} \leq c_{\fF_{a,b}}^{-1} {\bigmod{S_{R_n}(y)}}$.  
Thus, it remains to estimate $\bigmod{S_{R_n}(y)}$ from above.  Observe that, starting from $y$, we can distinguish the points $v$ in $S_{R_n}(y)$ 
according to the following: 
\begin{enumerate}
	\item[\itemno1]
		$v$ is in $\fF_a$, and the intersection of the geodesics $[y,v]$ and $[o,y]$ is $\{y\}$; 
	\item[\itemno2]
		$v$ is in $\fF_a$, and the intersection of the geodesics $[y,v]$ and $[o,y]$ is the geodesic $[p^{k}(y),y]$ for some $k$ in $[1,d(y,o)]$;
	\item[\itemno3]
		$v$ is in $\fF_b$.
\end{enumerate}
Therefore
\begin{equation} \label{f: counterex noncentred intermediate}
\begin{aligned}
\bigmod{S_{R_n}(y)}
	& = a^{R_n} + \sum_{k=1}^{d(y,o)} \, (a-1)\, a^{R_n-k-1} + b^{R_n-d(y,o)-1} \\
	& = a^{R_n-1} \, \big(1 - a^{-d(y,o)}+a\big) + b^{R_n-d(y,o)-1}. 
\end{aligned}
\end{equation}
Since $d(o,y) = \lfloor (2\pab - 1)\, n\rfloor-\lfloor \pab \, n\rfloor$, we have 
$R_n-d(y,o)-1 \leq n+2$.  Furthermore $a^{R_n} \leq a^{\pab \, n+ 2} = a^2\, b^n$.  These estimates, together with 
\eqref{f: counterex noncentred intermediate}, imply that 
$$
\bigmod{S_{R_n}(y)}
\leq \big(2\, a^2+b^2\big)\, b^n.
$$
Altogether, we can conclude that 
$$
	\cN \One_{E_n} (x) 
	\geq c_{\fF_{a,b}}\, \frac{\bigmod{E_n}}{\big(2\, a^2+b^2\big)\, b^n}
	= c_{\fF_{a,b}}\, \frac{b^{n-1}}{\big(2\, a^2+b^2\big)\, b^n} 
	\geq \frac{c_{\fF_{a,b}}}{\big(2\, a^2+b^2\big)\, b} 
$$
for every $x$ in $F_n$.   
Now, set $\ds c := \frac{c_{\fF_{a,b}}}{\big(2\, a^2+b^2\big)\, b}$.
Since $\cN \One_{E_n} \geq c$ on~$F_n$,
$$
	\frac{\bignormto{\cN \One_{E_n}}{p}{p}}{\bignormto{\One_{E_n}}{p}{p}} 
	\geq c^p \, \frac{\bigmod{F_n}}{\bigmod{E_n}}.
$$
If $\cN$ were bounded on $\lp{\fF_{a,b}}$ for some $p$ in $[1,\infty)$, 
then there would exist a constant $C$, independent of $n$, such that $\bigmod{F_n} \leq C\, \bigmod{E_n}$.  
By the definition of $\pab$, 
$$
\bigmod{F_n} 
= a^{\lfloor (2\pab - 1)\, n\rfloor}
\geq \frac{1}{a} \, \Big(\frac{b^{2}}{a}\Big)^n.
$$ 
Therefore the following should hold
$$
\frac{1}{a} \, \Big(\frac{b^{2}}{a}\Big)^n
\leq C\, b^{n},
\qquad\hbox{equivalently}\qquad
\Big(\frac{b}{a}\Big)^n
\leq C
$$
for all large $n$, which is a contradiction, because $a<b$.  

\medskip
Next we prove \rmii.  Recall that $\fF_{a,b} = \fF_a \cup \fF_b$.  For $x$ in $\fF_{a,b}$, set $\vr_b(x) := d(x, \fF_b)$ and $\vr_a(x) := d(x, \fF_a)$.

For any nonnegative $f$ in $\lu{\fF_{a,b}}$, write $f = f_a + f_b$, where $f_a := f\, \One_{\fF_a}$ and $f_b := f \, \One_{\fF_b}$.  Since
$$
\big\{\cM f> \la\big\}
\subseteq \big\{\cM f_a> \la/2\big\}
\cup \big\{\cM f_b> \la/2\big\},
$$
it suffices to prove that there exists a constant $C$ such that 
\begin{equation} \label{f: estimate fa}
	\bigmod{\big\{x\in \fF_{a,b}: \cM f_a(x) > \la/2\big\}}
	\leq \frac{C}{\la} \bignorm{f_a}{\lu{\fF_{a,b}}}
\quant \la >0,
\end{equation} 
and that a similar estimate holds for $f_b$.  

In order to prove \eqref{f: estimate fa}, we estimate $\cM f_a(x)$ when $x\in \fF_a$ and $x\in \fF_b$ separately.   
Suppose first that $x\in \fF_a$.   It is straightforward to check that 
$$
\bigmod{B_R(x)}
\geq V^a(R)  
\quant R\geq 0,
$$
and $\bigmod{B_R(x)} = V^a(R)$ if $R \leq \vr_b(x)$.
Therefore 
$$
\frac{1}{\bigmod{B_R(x)}} \, \int_{B_R(x)} f_a \wrt\mu 
\leq \frac{1}{V^a(R)} \, \int_{B_R(x)} f_a \wrt\mu 
\quant x \in \fF_a.  
$$
Consider the homogeneous tree $\fT_a$, a reference point $o_a$ in it, and the neighbours $w_0,\ldots, w_a$ of $o_a$ in $\fT_a$. 
The subset $\fT_a \setminus E(w_0)$, where $E(w_0)$ 
is the branch of $\fT_a$ of all points $y$ such that the geodesic $[o_a,y]$ starts with $[o_a,w_0]$, is isometric to $\fF_a$ via an isometry
$\cJ$ that maps $o$ to $o_a$, and the neighbours $g_1,\ldots, g_a$ of $o$ in $\fF_{a,b}$ to the neighbours $w_1,\ldots, w_a$ of $o_a$ in $\fT_a$.

Consider the function $f_a^\sharp$ on $\fT_a$ that vanishes on $E(w_0)$ and is equal to $f_a(x)$ at the point $\cJ(x)$. 
Then the average on the right hand side of the estimate above agrees with the average over the ball with centre $\cJ (x)$ and radius $R$ 
of the function $f_a^\sharp$ on $\fT_a$.  
Thus, $\cM f_a (x) \leq \cM f_a^\sharp (\cJ(x))$ for every $x$ in $\fF_a$.  
Since $\bignorm{f_a^\sharp}{\lu{\fT_a}} = \bignorm{f_a}{\lu{\fF_a}}$ and the centred maximal function on $\fT_a$ is of weak type $(1,1)$,
\begin{equation} \label{f: intermediate weak connected sum}
\begin{aligned}
\bigmod{\{x\in \fF_a: \cM f_a (x) > \la/2\}}
	& \leq \bigmod{\{y\in \fT_a: \cM f_a^\sharp (y) > \la/2\}} \\
	& \leq  \frac{C}{\la}\, \bignorm{f_a^\sharp}{\lu{\fT_a}} \\  
	& =     \frac{C}{\la}\, \bignorm{f_a}{\lu{\fF_a}}.  
\end{aligned}
\end{equation}

Suppose now that $x\in \fF_b$.  Observe that if $R \leq \vr_a(x)$, then $\bigmod{B_R(x)} = V^b(R)$.
Furthermore, if $R > \vr_a(x)$, then $B_R(x)$ contains all points in $\fF_b$, at distance $R$ from $x$.  The number of such points is equal
to $b^R$, and notice that $\ds b^R \geq \frac{b-1}{b+1} \, V^b(R)$, as a straightforward computation shows (see \eqref{f: volume profile}).  Therefore  
$$
\begin{aligned}
\cM f_a (x) 
	& \leq \frac{b+1}{b-1} \, \sup_{R \geq \vr_a(x)} \frac{1}{V^b(R)} \, \int_{\fF_a\cap B_{R-\vr_a(x)}(o)} f_a \wrt \mu \\
	& \leq \frac{b+1}{b-1} \,\frac{1}{V^b\big(\vr_a(x)\big)}\, \bignorm{f_a}{\lu{\fF_a}}
	\quant x \in \fF_b,
\end{aligned}
$$
because the support of $f_a$ is contained in $\fF_a$.  Hence 
$$
\big\{x\in \fF_b: \cM f_a (x) > \la/2 \big\}
\subseteq \Big\{x\in \fF_b: V^b\big(\vr_a(x)\big) < 2\, \frac{b+1}{b-1} \,\frac{\bignorm{f_a}{\lu{\fF_{a,b}}}}{\la}\Big\},
$$
so that 
$$
\bigmod{\big\{x\in \fF_b: \cM f_a (x) > \la/2 \big\}}
\leq 2\, \frac{b+1}{b-1} \,\frac{\bignorm{f_a}{\lu{\fF_{a,b}}}}{\la}.
$$
By combining the estimate above and \eqref{f: intermediate weak connected sum}, we conclude that 
$$
\bigmod{\big(\big\{x\in \fF_{a,b}: \cM f_a (x) > \la/2\big\}}
\leq \frac{C}{\la}\, \bignorm{f_a}{\lu{\fF_{a,b}}},
$$
as required.

\smallskip
Next we consider $f_b$.  We proceed much as above for $f_a$, and estimate $\cM f_b(x)$ when $x\in \fF_b$, and $x\in \fF_a$ separately.   

Suppose first that $x\in \fF_b$.  Clearly $\bigmod{B_R(x)} = V^b(R)$ in the case where $R \leq \vr_a(x)$, and 
$\ds \bigmod{B_R(x)} \geq b^R \geq \frac{b-1}{b+1} \, V^b(R)$ if $R > \vr_a(x)$.
Therefore $\ds\bigmod{B_R(x)} \geq \frac{b-1}{b+1} \, V^b(R)$ for every $R$, whence 
$$
\frac{1}{\bigmod{B_R(x)}} \, \int_{B_R(x)} f_b \wrt\mu 
\leq \frac{b+1}{b-1} \,\frac{1}{V^b(R)} \, \int_{\fF_b\cap B_R(x)} f_b \wrt\mu
\quant x \in \fF_b.  
$$
We can identify $\fF_b$ isometrically to a subset of the homogeneous tree~$\fT_b$ by identifying the vertex $g_0$ in $\fF_{a,b}$
with the reference point $o_b$ in $\fT_b$, and removing from $\fT_b$ the set $E(z_0)$
(here $z_0,\ldots, z_b$ are the neighbours of $o_b$ in $\fT_b$).  Denote by $\cJ$ this isometry, and by $f_b^\sharp$ the function on 
$\fT_b$ that vanishes on $E(z_0)$ and is equal to $f_b(x)$ at the point $\cJ(x)$. 

Then the average on the right hand side of the estimate above agrees with the average over the ball with centre $\cJ (x)$ and radius $R$ of the
function~$f_b^\sharp$ on $\fT_b$.  Thus, $\ds \cM f_b (x) \leq \frac{b+1}{b-1} \,\cM f_b^\sharp (\cJ(x))$ for every $x$ in $\fF_b$.  
Since $\bignorm{f_b^\sharp}{\lu{\fT_b}} = \bignorm{f_b}{\lu{\fF_b}}$, and the centred maximal function on $\fT_b$ is of weak type $(1,1)$,
\begin{equation} \label{f: intermediate weak connected sum II}
\begin{aligned}
\bigmod{\{x\in \fF_b: \cM f_b (x) > \la/2\}}
	& \leq \bigmod{\{y\in \fT_b: \cM f_b^\sharp (x) > \la/2\}} \\
	& \leq  \frac{C}{\la}\, \bignorm{f_b^\sharp}{\lu{\fT_b}} \\  
	& =     \frac{C}{\la}\, \bignorm{f_b}{\lu{\fF_b}}.  
\end{aligned}
\end{equation}

Finally, suppose that $x$ is in $\fF_a$.  Clearly $\bigmod{B_R(x)} = V^a(R)$ in the case where $R \leq \vr_b(x)$, and 
$\ds \bigmod{B_R(x)} \geq V^a\big(\vr_b(x)\big)$ if $R > \vr_b(x)$.  Thus, 
$$
\begin{aligned}
\cM f_b (x) 
& = \sup_{R > \vr_b(x)} \frac{1}{\bigmod{B_R^a(x)}} \, \int_{\fF_b\cap B_{R-\vr_a(x)}(o)} f_b \wrt \mu \\
& \leq \frac{C}{V^a\big(\vr_b(x)\big)} \, \bignorm{f_b}{\lu{\fF_b}}
\quant x \in \fF_a.  
\end{aligned}
$$
Hence 
$$
\big\{x\in \fF_a: \cM f_b (x) > \la/2 \big\}
\subseteq \Big\{x\in \fF_a: V^a\big(\vr_b(x)\big) < \frac{\bignorm{f_b}{\lu{\fF_{a,b}}}}{\la}\Big\},
$$
so that 
$$
\bigmod{\big\{x\in \fF_a: \cM f_b (x) > \la/2 \big\}}
\leq \frac{\bignorm{f_b}{\lu{\fF_{a,b}}}}{\la}.
$$
By combining the estimate above and \eqref{f: intermediate weak connected sum II}, we conclude that 
$$
\bigmod{\big(\big\{x\in \fF_{a,b}: \cM f_b (x) > \la/2\big\}}
\leq \frac{C}{\la}\, \bignorm{f_b}{\lu{\fF_{a,b}}},
$$
as required.  
\end{proof}

We end this section with a slight variant of \cite[Theorem~5.1]{V}.

\begin{theorem} \label{t: uncentred lower bound}
Suppose that $1\leq a \leq b$, that $b\geq 2$, and that $\fT$ is a tree in the class $\Upsilon_{a,b}$.  Assume that there exists a positive constants 
$C$ such that $\bigmod{B_r(x)} \geq C \, b^r$ for all vertices $x$ and nonnegative integers $r$.   Then the operator $\cN$ is of restricted weak type $(2,2)$.   
\end{theorem}

\begin{proof}
Suppose that $x$ is in $\fT$, and $B_R$ is a ball of radius $R$ containing $x$.  Then  
\begin{equation} \label{f: uncentred lower bound}
\frac{1}{\bigmod{B_R}}\, \int_{B_R} \bigmod{f} \wrt \mu 
\leq \frac{1}{C\, b^R}\, \int_{B_R} \bigmod{f} \wrt \mu 
\leq \frac{1}{C\, b^R}\, \int_{B_{2R}(x)} \bigmod{f} \wrt \mu.  
\end{equation}
As explained just above Theorem~\ref{t: main ab}, $\fT$ is isometric to a subtree of $\fT_b$, so that we can actually assume that $\fT$
is contained in $\fT_b$.  Hence any function $f$ on $\fT$ can be trivially extended to a function~$f^\sharp$ on $\fT_b$ as follows
$$
	f^\sharp (x) 
	:= f(x) \, \One_\fT (x) 
	\quant x \in \fT_b.
$$
Notice that $\bignorm{f}{\lorentz{p}{q}{\fT}} = \bignorm{f^\sharp}{\lorentz{p}{q}{\fT_b}}$ for all indices $p$ and $q$ such that $1\leq p,q\leq \infty$.
Now, observe that $b^R \geq c\, V^b(2R)^{1/2}$ for some positive constant~$c$.  This and \eqref{f: uncentred lower bound} imply that
$$
\frac{1}{\bigmod{B_R}}\, \int_{B_R} \bigmod{f} \wrt \mu 
\leq \frac{C}{V^b(2R)^{1/2}}\, \int_{B_{2R}^b(x)} \bigmod{f^\sharp} \wrt \mu.  
$$
By taking the supremum with respect to $R$ in $\BN$, we see that 
$$
	\cN f(x) 
	\leq C \, \cM_{2}^b f^\sharp (x) 
	\quant x \in \fT. 
$$
Therefore
$$
	\begin{aligned}
	C \bignorm{\cN f}{\lorentz{2}{\infty}{\fT}}
		& \leq  \bignorm{\cM_{2}^b f^\sharp}{\lorentz{2}{\infty}{\fT}} \\
		& \leq  \bignorm{\cM_{2}^b f^\sharp}{\lorentz{2}{\infty}{\fT_b}} \\
		& \leq  \bigopnorm{\cM_{2}^b}{\lorentz{2}{1}{\fT_b};\lorentz{2}{\infty}{\fT_b}} \bignorm{f^\sharp}{\lorentz{2}{1}{\fT_b}} \\
		& \leq  \bigopnorm{\cM_{2}^b}{\lorentz{2}{1}{\fT_b};\lorentz{2}{\infty}{\fT_b}} \bignorm{f}{\lorentz{2}{1}{\fT}};
	\end{aligned}
$$
the third inequality above follows from Veca's result \cite[Theorem~5.1]{V}. 

This concludes the proof of the theorem.  
\end{proof}

\section{Robustness of the results}
\label{s: Robustness}
The methods developed so far seem to be nonamenable to a straightforward extension to graphs.  However, simple examples show that
there are graphs, which are not trees, where the analogues of the results so far proved for trees hold.  This raises the question of
finding reasonably general classes of graphs, which share with trees properties similar to those discussed in the previous sections. 

Our analysis in this section is based on the notion of rough isometry, in the sense of M.~Kanai~\cite{K} 
(see Definition~\ref{def: RI} below).   

First, we need more notation and terminology.  We consider only \textit{simple graphs} $\fG$, i.e. undirected graphs without self-loops and multi-edges.  
There is a natural distance on simple graphs: $d(x,y)$ denotes the length of the shortest path joining $x$ and $y$.  Note that, contrary to 
what happens on trees, there may be many paths of minimal length joining any two points. 
For any~$x$ in $\fG$, the ball and the sphere with centre~$x$ and radius $r$ will be denoted by $B_r(x)$ and $S_r(x)$.  

We shall consider only connected graphs with \textit{$(1,Q)$-bounded geometry}, i.e. we assume that 
\begin{equation} \label{f: bdd geometry G}
	2\leq \nu(x) \leq Q+1 \quant x \in \fG.   
\end{equation}
We need a few simple geometric properties of $\fG$, established below.
For the sake of notational convenience, for every nonnegative integer $m$, we set 
$$
\ds\Om_{m,Q} := \frac{Q^{m+1}-1}{Q-1}.
$$
Since each vertex in $\fG$ has at most $Q+1$ neighbours, the volume of balls of radius $r$ in~$\fG$ is controlled by $V^Q(r)$, i.e. the volume of
balls with the same radius in~$\fT_Q$.  Thus,   
\begin{equation}\label{f: estball}
|B_{r}(x)|
\leq V^Q(r)
\quant x \in \fG \quant r \in \BN.
\end{equation} 
Also note that 
\begin{equation}\label{f: compball}
|B_{r+n}(x)| \leq  \Om_{n,Q}\,  |B_r(x)|
\quant n \in \BN.
\end{equation}
Indeed,
$$
|B_{r+n}(x)|
=    |B_r(x)|+\sum_{j=r+1}^{r+n} |S_j(x)|  
\leq |B_r(x)|+|S_r(x)| \, \sum_{j=1}^{n} Q^{j};
$$
now \eqref{f: compball} follows directly from this and trivial inequality $|S_r(x)| \leq |B_r(x)|$.

Notice that if $x$ and $y$ are points of $\fG$ at distance at most $K$, then 
\begin{equation}\label{f: compball II}
	\Om_{K,Q}^{-1} \,\,  |B_r(x)| \leq |B_{r}(y)| \leq  \Om_{K,Q} \,\,  |B_r(x)| 
\quant r \in \BN.
\end{equation}
Indeed, $B_r(y) \subseteq B_{r+K}(x)$, by the triangle inequality, so that
$$
|B_r(y)| \leq |B_{r+K}(x)| \leq  \Om_{K,Q}\,\,  |B_r(x)|;
$$
the last inequality follows from \eqref{f: compball}.  By reversing the role of $x$ and $y$, we obtain
$$
|B_r(x)| \leq  \Om_{K,Q}\,\,  |B_r(y)|;
$$
the left inequality in \eqref{f: compball II} follows directly from this.  

\begin{definition} \label{def: RI}
Suppose that $\fG$ and $\fG'$ are two connected simple graphs, with distances $d$ and $d'$, respectively.
A map $\vp : \fG \to \fG'$ is a \textit{strict rough isometry} if there exist positive numbers $K$ and $\be$ such that 
\begin{equation} \label{f: RI uno}
	K:= \max\, \big\{d'\big(\vp(\fG),x'\big): x' \in \fG' \big\}
\end{equation}
is finite, and
\begin{equation} \label{f: RI due}
	d(x,y)-\be \le d'\big(\vp(x),\vp(y)\big) \le d(x,y)+\be
	\quant x,y \in \fG.
\end{equation}
\end{definition} 

Points in $\fG$ and in $\fG'$ will be customarily denoted by \textit{nonprimed} and \textit{primed} lower case latin letters, respectively.
If $x'$ is a vertex in $\fG'$, we denote by $B_r(x')$ and $S_r(x')$ the ball and the sphere with centre~$x'$ and radius $r$. 
Without loss of generality we shall assume that $\be$ is a positive integer.  The next lemma contains a few simple geometric properties of $\vp$.

\begin{lemma} \label{l: simple prop vp}
	Suppose that $\vp$ is a strict rough isometry between $\fG$ and $\fG'$ with $(1,Q)$ and $(1,Q')$-bounded geometry, respectively.  The following hold:
	\begin{enumerate}
		\item[\itemno1]
			$\bigmod{\vp^{-1} (y')} \leq V^Q(\be)$ for every $y'$ in $\vp(\fG)$;  
		\item[\itemno2]
			$\{B_K\big(\vp(x)\big): x \in \fG\}$ is a covering of $\fG'$  with the finite overlapping property.  Furthermore
			\begin{equation} \label{f: combined preimage}
			\bigmod{\big\{x\in \fG: z'\in B_K\big(\vp(x)\big)\big\}}
			\leq V^Q(\be) \,\, V^{Q'}(K).   
			\end{equation} 
	\end{enumerate}
\end{lemma}

\begin{proof}
To prove \rmi, observe that if $v$ and $w$ are points in $\fG$ such that $\vp(v) = \vp(w)$, then the left inequality in \eqref{f: RI due} yields
$d(v,w) \leq \be$.  Thus, $w$ belongs to $B_\be(v)$, and \rmi\ follows directly from \eqref{f: estball}.

Next we prove \rmii.  By \eqref{f: RI uno}, the sequence $\{B_K\big(\vp(x)\big): x \in \fG\}$ is a covering of $\fG'$.

Now suppose that $z'$ is a point in $\fG'$.  If $z'$ belongs to $B_K\big(\vp(x)\big)$, then $\vp(x)$ belongs to $B_K(z')$.  
Since $\bigmod{B_K(z')} \leq V^{Q'}(K)$, there are at most $V^{Q'}(K)$ balls of the given covering containing $z'$.  Then \eqref{f: combined preimage}
follows from this and~\rmi. 
\end{proof}


The next theorem is the main result of this section.  Preliminarily, we need more notation.  Suppose that $\vp: \fG\to\fG'$ is a strict
rough isometry, as in Definition~\ref{def: RI}, and that $f$ is a function on $\fG'$.  Then $\pi f$ will denote the function on $\vp(\fG)$, defined by 
\begin{equation} \label{f: pi f}
	(\pi f) (x')
	:= \int_{B_K(x')} \mod{f} \wrt \mu
	\quant x'\in \vp(\fG).  
\end{equation}
Given a function $f$ on $\fG'$, and a nonnegative integer $R$, consider the functions~$A_R f$ and $\cA_R f$, defined by 
$$
	A_R f(z')
	:= \frac{1}{\bigmod{B_R(z')}} \, \int_{B_R(z')} \mod{f} \wrt \mu
	\quant z' \in \fG',
$$
and 
$$
	\cA_R f(z')
	:= \sup_{B\in \fB_R: B \ni z'} \, \frac{1}{\bigmod{B}} \, \int_{B} \mod{f} \wrt \mu
	\quant z' \in \fG':  
$$
here $\fB_R$ denotes the collection of all balls of radius $R$ in $\fG'$.

\begin{theorem} \label{t: quasi iso} 
Suppose that $\fG$ and $\fG'$ are connected simple graphs with $(2,Q)$ and $(2,Q')$-bounded geometry, respectively, and that $\vp: \fG\to \fG'$ is 
a strict rough isometry.  Assume that $1\leq p<\infty$, and that $1\leq r\leq s\leq \infty$.  The following hold: 
\begin{enumerate}
	\item[\itemno1] 
		if $\cM$ is bounded from $\lorentz{p}{r}{\fG}$ to $\lorentz{p}{s}{\fG}$, then $\cM$ is bounded from $\lorentz{p}{r}{\fG'}$
		to $\lorentz{p}{s}{\fG'}$;
	\item[\itemno2] 
		if $\cN$ is bounded from $\lorentz{p}{r}{\fG}$ to $\lorentz{p}{s}{\fG}$, then $\cN$ is bounded from $\lorentz{p}{r}{\fG'}$
		to $\lorentz{p}{s}{\fG'}$;
\end{enumerate}
\end{theorem} 

\begin{proof}
In this proof $\mu$ will denote both the counting measure on $\fG$ and on~$\fG'$.  

First we prove \rmi.   
We consider the maximal operators 
\begin{equation} \label{f: MK I}
	\cM^K\!  f (z')
	:= \max_{R\leq K} \, A_R f(z')
	\quant z' \in \fG'
\end{equation} 
and 
$$
	\cM_K f(z') 
	:= \sup_{R>K} \, A_R f(z')
	\quant z' \in \fG',  
$$
where $K$ is as in Definition~\ref{def: RI}.
Observe that $\cM f \leq \cM^K\! f + \cM_K f$.  The required boundedness properties of $\cM$ will follow directly from the corresponding properties
of $\cM^K$ and $\cM_K$, which will be established in \textit{Step I} and \textit{Step~III} below, respectively.  

\textit{Step I: boundedness of $\cM^K$}.  
Observe that, trivially, $\ds \cM^K \! f \leq \sum_{R=0}^K \, A_R f$, so that it suffices to prove that $A_R f$ is bounded from $\lorentz{p}{r}{\fG'}$
to $\lorentz{p}{s}{\fG'}$ for $R=0,\ldots,K$.  

Suppose that $\al>0$ and $A_Rf(z') >\al$.  Then there exists a point $w'$ in $B_R(z')$ such that $\bigmod{f(w')} > \al$ (for otherwise $A_Rf(z') \leq \al$).
Thus, we have a map from $E_{A_Rf}(\al)$ to $E_f(\al)$ that associates to $z'$ in $E_{A_Rf}(\al)$ a point~$w'$ in $E_f(\al)$.  This map 
may very well be not injective.
However, a point~$w'$ in~$E_f(\al)$ can be, at most, the image of all the points in the ball $B_R(w')$.  Therefore
\begin{equation} \label{f: distr function}
\bigmod{E_{A_Rf}(\al)}
\leq V^{Q'}\!(R) \, \bigmod{E_f(\al)}.
\end{equation}
Fr{}om the definition of the Lorentz (quasi-) norm, and the contractivity of the injection $\lorentz{p}{r}{\fG'} \subseteq \lorentz{p}{s}{\fG'}$  
we deduce that 
$$
	\bignorm{A_R f}{\lorentz{p}{s}{\fG'}} 
	\leq V^{Q'}\!(R)^{1/p} \bignorm{f}{\lorentz{p}{s}{\fG'}}
	\leq V^{Q'}\!(R)^{1/p} \bignorm{f}{\lorentz{p}{r}{\fG'}}\!.
$$
Hence 
\begin{equation} \label{f: est MsupK}
\bignorm{\cM^K \! f}{\lorentz{p}{s}{\fG'}}  
\leq \sum_{R=0}^K \, \bignorm{A_R f}{\lorentz{p}{s}{\fG'}} 
\leq 2^{1/p} \,\, \Om_{K+2,Q'}^{1/p} \bignorm{f}{\lorentz{p}{r}{\fG'}}\!,
\end{equation}
as required.  

%

\medskip
\textit{Step II}.  
Consider a ball $B_R(z')$ in $\fG'$, with $R>K$.  By Lemma~\ref{l: simple prop vp}~\rmii,
$\{B_K\big(\vp(x)\big): x \in \fG\}$ is a covering of $\fG'$.  Therefore there exists $z$ in $\fG$ such that $d'\big(\vp(z),z'\big) \leq K$
(thus, in particular, $\vp(z)$ belongs to $B_R(z')$).  

We \textit{claim} that 
\begin{equation} \label{f: claim rough iso}
\frac{1}{\bigmod{B_R(z')}} \,\int_{B_R(z')} \mod{f} \wrt \mu 
\leq C_{\be,K} \, \frac{1}{\bigmod{B_{R'}(z)}} \, \int_{B_{R'}(z)} \!\!\big((\pi f)\circ \vp\big) \wrt \mu,
\end{equation}
where $R'= R+2K+\be$, and $\ds C_{\be,K}= V^Q\!(\be) \, \Om_{2K+\be,Q}\, \Om_{\be,Q'}\,\Om_{K,Q'}$.  

To prove the claim, first observe that 
\begin{equation} \label{f: covering K I}
\big\{y\in \fG: B_K\big(\vp(y)\big)\cap B_R(z') \neq \emptyset \}
\subseteq B_{R+2K+\be}(z).   
\end{equation}
Indeed, if $B_K\big(\vp(y)\big)\cap B_R(z') \neq \emptyset$, then $d'\big(\vp(y),z'\big) \leq R+K$, whence 
$$
d(y,z)
\leq d'\big(\vp(y),\vp(z)\big) + \be
\leq d'\big(\vp(y),z') + d'\big(z',\vp(z)\big) + \be
\leq R+2K+\be,
$$
as required. 

Set 
$$
E(z',R) := \big\{y'\in \vp (\fG): B_K(y')\cap B_R(z')  \neq \emptyset\big\}.
$$
By \eqref{f: covering K I} 
\begin{equation} \label{f: covering K II}
\vp^{-1} \big(E(z',R)\big)
\subseteq B_{R+2K+\be}(z).   
\end{equation}
Since $\{B_K\big(y'\big): y' \in \vp(\fG)\}$ is a covering of $\fG'$,  
$
\ds B_R(z') \subseteq \bigcup_{y'\in E(z',R)} \, B_K (y').
$
Therefore
\begin{equation} \label{f: intermediate mean I}
\begin{aligned}
\int_{B_R(z')} \mod{f} \wrt \mu 
	& \leq \sum_{y'\in E(z',R)} \, \int_{B_K(y')} \mod{f} \wrt \mu \\
	& =    \sum_{y'\in E(z',R)} \, (\pi f)(y')  \\
	& \leq \sum_{y\in B_{R+2K+\be}(z)} \, \big[(\pi f)\circ\vp\big](y);
\end{aligned}
\end{equation}
the last inequality follows from \eqref{f: covering K II}.
%
We need to estimate the ratio $\bigmod{B_{R+2K+\be}(z)}/\bigmod{B_R(z')}$.  By \eqref{f: compball},
$$
\bigmod{B_{R+2K+\be}(z)}
\leq \Om_{2K+\be,Q} \, \bigmod{B_{R}(z)}.
$$
Since $\vp$ is at most $V^Q(\be)$-to-one, $\bigmod{B_{R}(z)} \leq V^Q(\be) \, \bigmod{\vp\big(B_{R}(z)\big)}$.  Now, if~$y'$
belongs to $\vp\big(B_{R}(z)\big)$, then $y'= \vp(y) $ for some $y$ in $B_R(z)$, and 
$$
d'\big(\vp(y),\vp(z)\big) 
\leq d(y,z)+\be 
\leq R+\be,
$$
so that $\vp\big(B_{R}(z)\big)$ is contained in $B_{R+\be}\big(\vp(z)\big)$.  Therefore 
$$
\begin{aligned}
\bigmod{B_{R}(z)} 
	& \leq V^Q\!(\be) \, \bigmod{B_{R+\be}\big(\vp (z)\big)}  \\
	& \leq V^Q\!(\be) \, \Om_{\be,Q'} \,\bigmod{B_{R}\big(\vp (z)\big)} \\
	& \leq V^Q\!(\be) \, \Om_{\be,Q'} \, \Om_{K,Q'}\,\bigmod{B_{R}(z')}:
\end{aligned}
$$
we have used \eqref{f: compball} and \eqref{f: compball II} in the last two inequalities.  Altogether, we have proved that   
$$
\frac{\bigmod{B_{R+2K+\be}(z)}}{\bigmod{B_R(z')}}
\leq C_{K,\be},
$$
where $C_{K,\be}$ is as in the claim.  
The claim follows directly from this and \eqref{f: intermediate mean I}. 

\medskip
\textit{Step III: conclusion of the proof of \rmi}. 
By \eqref{f: claim rough iso}, to each $z'$ in $\fG'$, we can associate a point~$z$ in $\fG$ such that $d'\big(z',\vp(z)\big) \leq K$ and 
$$
\cM_K f (z')
\leq C_{K,\be}\,\, \cM \big((\pi f)\circ \vp\big) (z).
$$
Suppose that $\al>0$ and that $z'$ is a point in $E_{\cM_K f}(\al)$.
Since the point $z$ can be associated to at most $V^{Q'}(K)$ points $z'$ in $\fG'$, 
$$
\bigmod{E_{\cM_K f}(\al)}
\leq V^{Q'}\!(K) \,  \bigmod{E_{\cM ((\pi f) \circ \vp)}(\al/C_{K,\be})}.
$$
Consequently
$$
\bignorm{\cM_K f}{\lorentz{p}{s}{\fG'}}
\leq V^{Q'}\!(K)^{1/p} \Big(p\ioty  \bigmod{E_{\cM ((\pi f) \circ \vp)}(\al/C_{K,\be})}^{s/p} \, \al^{s-1} \wrt \al\Big)^{1/s}.\, 
$$
Changing variables and using the definition of the (quasi-) norm of $\lorentz{p}{s}{\fG'}$, we see that 
$$
\bignorm{\cM_K f}{\lorentz{p}{s}{\fG'}}
\leq C_{K,\be}\, V^{Q'}\!(K)^{1/p} \bignorm{\cM \big((\pi f) \circ \vp\big)}{\lorentz{p}{s}{\fG}}.  
$$
Since, by assumption, $\cM$ is bounded from $\lorentz{p}{r}{\fG}$ to $\lorentz{p}{s}{\fG}$, 
$$
\bignorm{\cM \big((\pi f) \circ \vp\big)}{\lorentz{p}{s}{\fG}}
\leq \bigopnormto{\cM}{L^{p,r}(\fG);L^{p,s}(\fG)} \, \bignorm{(\pi f) \circ \vp}{\lorentz{p}{r}{\fG}}.
$$
Note that $\bigmod{(\pi f)\circ \vp(z)} > \al$ if and only if $\ds \int_{B_K(\vp(z))} \, \bigmod{f} \wrt \mu > \al$,
equivalently if and only if $A_Kf\big(\vp(z)\big)>\al/\bigmod{B_K\big(\vp(z)\big)}$.  Thus,
$$
\begin{aligned}
	\bigmod{E_{(\pi f)\circ \vp}(\al)}
	& =    \bigmod{\big\{z\in \fG: A_Kf\big(\vp(z)\big)>\al/\bigmod{B_K\big(\vp(z)\big)}\big\}}  \\
	& \leq \bigmod{\big\{z\in \fG: A_Kf\big(\vp(z)\big)>\al/V^{Q'}\!(K)\}}  \\
	& \leq V^Q\!(\be)\, \bigmod{\big\{z'\in \fG': A_Kf(z')>\al/V^{Q'}\!(K)\}};
\end{aligned}
$$
the second inequality above follows from the obvious estimate $\bigmod{B_K\big(\vp(z)\big)}\leq V^{Q'}\!(K)$, and the third
from the fact that $\vp$ is, at most, $V^Q\!(\be)$-to-one.  Now, \eqref{f: distr function}, with $K$ in place of $R$, yields the estimate
$$
\bigmod{\big\{z'\in \fG': A_Kf(z')>\al/V^{Q'}\!(K)\}}
\leq V^{Q'}\!(K) \, \bigmod{E_f(\al/V^{Q'}\!(K))}.  
$$
By combining these estimates, we see that 
$$
\bigmod{E_{(\pi f)\circ \vp}(\al)}
\leq V^Q\!(\be)\, V^{Q'}\!(K) \, \bigmod{E_f(\al/V^{Q'}\!(K))},   
$$
whence 
$$
\bignorm{(\pi f) \circ \vp}{\lorentz{p}{r}{\fG}}
\leq V^{Q'}\!(K)^{1+1/p} \,  V^Q(\be)^{1/p} \bignorm{f}{\lorentz{p}{r}{\fG'}}.  
$$
Altogether, we obtain the bound
$$
\bignorm{\cM_K f}{\lorentz{p}{s}{\fG'}}
\!\leq C_{K,\be} V^{Q'}\!(K)^{1+2/p} \, V^Q\!(\be)^{1/p} \, \bigopnorm{\cM}{L^{p,r}(\fG);L^{p,s}(\fG)} \!\bignorm{f}{\lorentz{p}{s}{\fG'}}\!\!,
$$
which, combined with \eqref{f: est MsupK}, yields \rmi.  

\smallskip
The proof of \rmii\ copies the strategy of the proof of \rmi; we shall only indicate the nontrivial changes needed and leave 
the other details to the interested reader. 
We consider the maximal operators 
\begin{equation} \label{f: MK I}
	\cN^K\!  f (z')
	:= \max_{R\leq K} \, \cA_R f(z')
	\quant z' \in \fG'
\end{equation} 
and 
$$
	\cN_K f(z') 
	:= \sup_{R>K} \, \cA_R f(z')
	\quant z' \in \fG',  
$$
where $K$ is as in Definition~\ref{def: RI}.  Observe that $\cN f \leq \cN^K\! f + \cN_K f$.  

To prove the boundedness of $\cN^K$ we proceed much as in the proof of Step I above.
Given $\al>0$ and a point $z'$ such that $\cA_Rf(z') >\al$, there exists a ball $B_R$ of radius $R$ containing~$z'$ such that 
$\ds \frac{1}{B_R} \, \int_{B_R} \mod{f} \wrt \mu > \al$.  Therefore there exists a point $w'$ in $B_R$ such that $\bigmod{f(w')} > \al$ 
(for otherwise the average above would be $\leq \al$).

Thus, we have a map from $E_{\cA_Rf}(\al)$ to $E_f(\al)$ that associates to $z'$ in $E_{\cA_Rf}(\al)$ a point~$w'$ in $E_f(\al)$.  
The point~$w'$ in~$E_f(\al)$ can be the image of, at most, all the points in the ball $B_{2R}(w')$ (rather than $B_R(w')$ as in Step I).  Therefore
\begin{equation} \label{f: distr function II}
\bigmod{E_{\cA_Rf}(\al)}
\leq V^{Q'}\!(2R) \, \bigmod{E_f(\al)}.
\end{equation}
Then we proceed, \textit{mutatis mutandis}, as in the proof of Step~I, using \eqref{f: distr function II} instead of \eqref{f: distr function}.

Next, let $B_R$ be a ball of radius $R>K$ containing $z'$.  
By Lemma~\ref{l: simple prop vp}~\rmii, there exists $z$ in $\fG$ such that $d'\big(\vp(z),z'\big) \leq K$
(thus, in particular, $\vp(z)$ belongs to $B_{K}(z')$).  Denote by $w'$ the centre of $B_R$.  Another application of Lemma~\ref{l: simple prop vp}~\rmii\
guarantees the existence of a point $w$ in $\fG$ with $d'\big(\vp(w),w'\big) \leq K$.  Thus, $\vp(z)$ belongs to $B_{R+2K}\big(\vp(w)\big)$.  
Set $R'= R+2K+\be$.  Since $\vp$ is a strict rough isometry, $d(z,w) \leq R'$.  

Thus, we associate to any ball $B_R$ of radius $R$ containing $z'$ a ball $B_{R'}$ in~$\fG$, containing $z$.   
Then, by arguing as in the proof of \eqref{f: claim rough iso}, we can prove that there exists a constant $C_{\be,K}'$ such that 
$$
\frac{1}{\bigmod{B_R}} \,\int_{B_R} \mod{f} \wrt \mu 
\leq C_{\be,K}' \, \frac{1}{\bigmod{B_{R'}}} \, \int_{B_{R'}} \!\!\big((\pi f)\circ \vp\big) \wrt \mu.  
$$
The conclusion then follows by arguing much as in the proof of Step~III above.  

This concludes the proof of \rmii, and of the theorem.
\end{proof}

Theorem~\ref{t: quasi iso} applies, for instance, to all graphs that are compact perturbations of $\fT_b$, i.e. to  graphs $\fG$ with the following property:
there exists compact sets $K\subset \fG$ and $K_b \subset \fT_b$ so that $\fG \setminus K=\fT_b\setminus K_b$.  

To prove this, fix a point $w$ in $K$, and consider the mapping $\vp: \fT_b \to \fG$, defined by 
\begin{enumerate}
	\item[\itemno1] $\vp(x) =x$  for every $x \in \fT_b \setminus K_b$;
	\item[\itemno2] $\vp(K_b) = w$.
\end{enumerate}
Observe that $\vp$ is a strict rough isometric between $\fT_b$ and $\fG$.  Indeed, $d\big(\varphi(\fT_b),x\big) \leq \diam(K)$ for all $x$ in $\fG$, and 
$$
d(x,y)-\eta \leq d\big(\varphi(x),\vp(y)\big) \le d(x,y)+\eta,
$$
where $\eta := \max\big(\diam(K),\diam(K_b)\big)$.


\begin{thebibliography}{MMV2}



\bibitem[CMS1]{CMS1} M. Cowling, S. Meda, A. G. Setti, 
A weak type $(1, 1)$ estimate for a maximal operator on a group of isometries of homogeneous trees,
\emph{Coll. Math.} \textbf{118} (2010), 223--232.

\bibitem[CMS2]{CMS2} M. Cowling, S. Meda, A. G. Setti, 
\emph{An overview of harmonic analysis on the group of isometries of a homogeneous tree}, Exposition. Math. {16} (1998), 385--423.
	

\bibitem[FTN]{FTN} A.~Fig\`a-Talamanca and C.~Nebbia, {\it Harmonic Analysis and
Representation Theory for Groups Acting on Homogeneous Trees}, London Math.
Society Lecture Notes Series, n.~\textbf{162},  Cambridge University Press, 1991.

\bibitem[K1]{K1} D.~Kosz, On relations between weak and strong type inequalities for maximal operators on non-doubling
metric measure spaces, 
\textit{Publ. Mat.} \textbf{62} (2018), 37--54. 

\bibitem[K2]{K2} D.~Kosz, On relations between weak and restricted weak type inequalities for maximal operators on
non-doubling metric measure spaces, 
\textit{Studia Math.} \textbf{241} (2018), 57--70. 

\bibitem[K3]{K3} D.~Kosz, Maximal operators on Lorentz spaces in non-doubling setting,
\textit{Math. Z.} \textbf{298} (2021), 1523--1543.


\bibitem[I1]{I1} A.D. Ionescu, 
A maximal operator and a covering lemma on non-compact symmetric spaces
\emph{Math. Res. Lett.} \textbf{7} (2000), 83--93.

\bibitem[I2]{I2}  A.D. Ionescu, Fourier integral operators on noncompact 
symmetric spaces of real rank one, \emph{J. Funct. Anal.} 
\textbf{174} (2000), 274--300.

\bibitem[K]{K} M. Kanai, Rough isometries, and combinatorial approximation of non-compact
Riemannian manifolds, 
\textit{J. Math. Soc. Japan} \textbf{37} (1985), 391--413.

\bibitem[LS]{LS} M. Levi and F. Santagati, Hardy--Littlewood fractional maximal operators on homogeneous trees, available at 
https://arxiv.org/abs/2211.11871

\bibitem[L1]{L1} H.-Q. Li, La fonction maximale de Hardy--Littlewood sur une classe d'espaces m\'etriques measurables, 
\textit{C. R. Acad. Sci. Paris, Ser. I} \textbf{338} (2004), 31--34.

\bibitem[L2]{L2} H.-Q. Li, La fonction maximale non centr\'ee sur les vari\'et\'es de type cuspidale, 
\textit{J. Funct. Anal.} \textbf{229} (2005), 155--183.

\bibitem[L3]{L3} H.-Q. Li, Les fonctions maximales de Hardy--Littlewood pour des measures sur les vari\'et\'es de type cuspidale, 
\textit{J. Math. Pures Appl.} \textbf{88} (2007), 261--275.

\bibitem[MZ]{MZ} A. M. Mantero and A. Zappa, The Poisson transform on free
groups and uniformly bounded representations, 
\textit{J. Funct. Anal.} \textbf{51} (1983), 372--399.



\bibitem[NT]{NT} A. Naor and T. Tao, Random martingales and localization of maximal inequalities, 
\emph{J. Funct. Anal.} \textbf{259} (2010), 731--779.


\bibitem[ORS]{ORS} S. Ombrosi, I.P. Rivera-Rios and M.D. Safe, Fefferman--Stein inequalities for the Hardy--Littlewood maximal function 
on the infinite rooted $k$-tree, 
\textit{Int. Math. Res. Not.} \textbf{4} (2021), 2736--2762.

\bibitem[OR]{OR} S. Ombrosi and I.P. Rivera-Rios, Weighted $L^p$ estimates on the infinite rooted $k$-ary tree,
\textit{Math. Ann.} \textbf{384} (2022), 491--510.


\bibitem[RT]{RT} R.~Rochberg and M.~Taibleson, Factorization of the Green's operator and weak-type estimates for a random walk on a tree,
\emph{Publ. Math.} \textbf{35} (1991), 187--207.


\bibitem[ST]{ST}
J. Soria and P. Tradacete, Geometric properties of infinite graphs and the Hardy--Littlewood maximal operator,
\emph{J. Anal. Math.} \textbf{137} (2019), 913--937.



\bibitem[SW]{SW} E.M. Stein and G. Weiss,
\emph{Introduction to Fourier Analysis on Euclidean Spaces},
Princeton University Press, 1971.

\bibitem[Str]{Str} J.-O.~Str\"omberg,
Weak type $L^1$ estimates for maximal functions on non-compact
symmetric spaces, \emph{Ann. of Math.}
\textbf{114} (1981), 115--126.

\bibitem[V]{V} A. Veca, The Kunze--Stein phenomenon on the isometry group of a tree,
{\it Bull. Austral. Math. Soc.} \textbf{65} (2002), 153--174. 

\bibitem[Wo]{Wo} R.K. Wojciechowski,
\emph{Stochastic Completeness of Graphs},
Ph.D thesis, arXiv:0712.1570 [math.SP].

\end{thebibliography}
\end{document}